\documentclass[a4paper,12pt]{amsart}

\usepackage[top=2.5cm, bottom=2.5cm, left=2.5cm, right=2.5cm]{geometry}

\usepackage{amssymb}
\usepackage{ascmac}
\usepackage[dvipdfmx]{graphicx}
\usepackage{color}

\usepackage[dvipdfm,colorlinks=true,bookmarks=true,
bookmarksnumbered=true,bookmarkstype=toc,linktocpage=true
]{hyperref}

\newtheorem{Def}{Definition}[section]
\newtheorem{Thm}[Def]{Theorem}
\newtheorem{Lem}[Def]{Lemma}
\newtheorem{Assumption}[Def]{Assumption}
\newtheorem{Rem}[Def]{Remark}
\newtheorem{Cor}[Def]{Corollary}
\newtheorem{Prop}[Def]{Proposition}

\newtheorem{Example}[Def]{Example}

\numberwithin{equation}{section}

\allowdisplaybreaks

\usepackage{ascmac}

\newcommand{\mca}{\mathcal{A}}

\newcommand{\mcc}{\mathcal{C}}

\newcommand{\mcf}{\mathcal{F}}

\newcommand{\mcl}{\mathcal{L}}

\newcommand{\mbbu}{\mathbb{U}}

\newcommand{\mbbg}{\mathbb{G}}
\newcommand{\mbbh}{\mathbb{H}}

\newcommand{\mbbr}{\mathbb{R}}
\newcommand{\mbbx}{\mathbb{X}}

\newcommand{\mbby}{\mathbb{Y}}

\newcommand{\al}{\alpha}
\newcommand{\del}{\delta}
\newcommand{\ep}{\epsilon}

\newcommand{\D}{\Delta}

\newcommand{\Sig}{\Sigma}
\newcommand{\lam}{\lambda}

\newcommand{\gam}{\gamma}
\newcommand{\Gam}{\Gamma}
\newcommand{\p}{\partial}

\newcommand{\cil}{\xrightarrow{\mcl}} 
\newcommand{\cip}{\xrightarrow{p}} 
\newcommand{\argmax}{\mathop{\rm argmax}}

\def\nn{\nonumber}

\def\sumj{\sum_{j=1}^{n}}
\def\intj{\int_{t_{j-1}}^{t_j}}

\def\tes{\hat{\theta}_{n}}
\def\aes{\hat{\alpha}_{n}}

\def\ges{\hat{\gamma}_{n}}
\def\bges{\hat{\gamma}_{n}^{\textbf{B}}}
\def\baes{\hat{\alpha}_{n}^{\textbf{B}}}

\def\pb{P^{\textbf{B}}}
\def\eb{E^{\textbf{B}}}

\title[Bootstrap method for misspecified ergodic L\'{e}vy driven SDE models]
{Bootstrap method for misspecified ergodic L\'{e}vy driven stochastic differential equation models}
\date{\today}

\author{Yuma Uehara}
\address[Department of Mathematics, Faculty of Engineering Science, Kansai University, 3-3-35 Yamate-cho, Suita-shi, Osaka 564-8680, Japan
]{}
\email{y-uehara@kansai-u.ac.jp}

\begin{document}

\maketitle

\begin{abstract}
In this paper, we consider possibly misspecified stochastic differential equation models driven by L\'{e}vy processes.
Regardless of whether the driving noise is Gaussian or not, Gaussian quasi-likelihood estimator can estimate unknown parameters in the drift and scale coefficients.
However, in the misspecified case, the asymptotic distribution of the estimator varies by the correction of the misspecification bias, and consistent estimators for the asymptotic variance proposed in the correctly specified case may lose theoretical validity. 
As one of its solutions, we propose a bootstrap method for approximating the asymptotic distribution.
We show that our bootstrap method theoretically works in both correctly specified case and misspecified case without assuming the precise distribution of the driving noise.
\end{abstract}

\section{Introduction}
We suppose that the data-generating structure is the following one-dimensional stochastic differential equation
\begin{equation}\label{dg}
dX_t=A(X_t)dt+C(X_{t-})dZ_t,
\end{equation} 
defined on a stochastic basis $(\Omega,\mcf,(\mcf_t)_{t\geq0},P)$ where 
\begin{itemize}
    \item the driving noise $Z$ can either be a standard Wiener process or a pure-jump L\'{e}vy process;
    \item $\mcf_t=\sigma(X_0)\vee\sigma(Z_s;s\leq t)$;
    \item The initial variable $X_0$ is $\mcf_0$-measurable.
\end{itemize}
In this paper, we consider the situation where high-frequency samples $\mbbx:=(X_{t_j})_{j=0}^n$ from the solution path $X$ are obtained in the so-called ``rapidly increasing design": $t_j=t_j^n:=jh_n$, $T_n:=nh_n\to\infty$, and $nh_n^2\to0$.
To deal with the effect of model misspecification being inevitable in statistical modeling, we consider the following parametric model on $(\Omega,\mcf,(\mcf_t)_{t\geq0},P)$:
\begin{equation}\label{Model}
dX_t=a(X_t,\al)dt+c(X_{t-},\gam)dZ_t.
\end{equation}
Here the functional forms of drift coefficient $a:\mbbr\times\Theta_\al\mapsto\mbbr,$ and scale coefficient $c:\mbbr\times\Theta_\gam\mapsto\mbbr$ are supposed to be known except for the drift parameter $\al$ and scale parameter $\gam$.
We also suppose that $\al$ and $\gam$ belong to bounded convex domains $\Theta_\al\subset\mbbr^{p_\al}$ and $\Theta_\gam\subset\mbbr^{p_\gam}$, respectively.
We note that the coefficients are possibly misspeciefied, that is, the parametric family $\{(a(\cdot,\al),c(\cdot,\gam)); \al\in\Theta_\al,\gam\in\Theta_\gam\}$ does not include the true coefficients $(A(\cdot),C(\cdot))$.
From now on, the terminologies ``misspecification", ``misspeficied" and ``correctly specified" will be used for the above meaning unless otherwise mentioned.
In this framework, there are four possible cases:
\begin{enumerate}
    \item Correctly specified diffusion case: the driving noise $Z$ is a standard Wiener process and the coefficients are correctly specified;
    \item Misspecified diffusion case: the driving noise $Z$ is a standard Wiener process and the coefficients are missspecified;
    \item Correctly specified pure-jump L\'{e}vy driven case: the driving noise $Z$ is a pure-jump L\'{e}vy process and the coefficients are correctly specified;
    \item Misspecified pure-jump L\'{e}vy driven case: the driving noise $Z$ is a pure-jump L\'{e}vy process and the coefficients are missspecified.
\end{enumerate}

For the estimation of $\theta$, we consider Gaussian quasi-likelihood estimation. It is a tractable and powerful tool for estimating mean and variance structure in the sense that we need not to assume the precise distribution of the error variable.
For various statistical models including \eqref{Model}, its theoretical property has been analyzed.
Especially, for correctly specified diffusion models, the asymptotic behavior of the Gaussian quasi maximum likelihood estimator (GQMLE) is verified for example by \cite{Yos92}, \cite{GenJac93}, and \cite{Kes97}.
As for correctly specified non-Gaussian L\'{e}vy driven SDE models, \cite{Mas13-1} clarified the theoretical property of the GQMLE.
Although its convergence rate is slower than the correctly specified diffusion case, it still has the consistency and asymptotic normality.
The papers \cite{UchYos11} and \cite{Ueh19} extended the results to the case where the drift and (or) scale coefficient are (is) misspecifed.
In these papers, the misspecification bias is handled by the theory of (extended) Poisson equation, and inevitably the asymptotic distribution of the GQMLE contains the solution.
Hence, the estimators of the asymptotic variance which have been proposed for the correctly specified case does not work in the misspecified case.
As a result, the confidence intervals and hypothesis testing based on the estimators no longer have theoretical validity in the misspecified case.
This is a serious problem since in practice, we cannot avoid the risk of model misspecification.
The primary object of this paper is to overcome this issue.

When it is tough to evaluate the asymptotic distribution of some statistic directly, bootstrap methods originally introduced by \cite{Efr79} often serve as a good prescription.
As for high-frequently observed settings, bootstrap methods also do and indeed for various purposes such as estimating realized volatility distribution \cite{GonSilMed09}, making statistical inference in jump regressions \cite{LiTodTauChe17},
executing jump tests \cite{DovGonSilHou19}, 
kinds of the methods have been proposed. 
In this paper, we follow this direction.
We construct a block bootstrap Gaussian quasi-score function which can uniformly approximate the asymptotic distribution of the GQMLE both in the correctly specified and misspecified case.
More specifically, we divide $\{1,\dots,n\}$ ($n$ denotes the sample size) into $k_n$ blocks, and for each block, we generate a bootstrap weight.
Based on the weights, we construct the bootstrap score function and estimator.
Furthermore, by introducing a adjustment term, our method can uniformly approximate the asymptotic distribution without specifying the distribution of the driving noise although the convergence rate of the scale parameter is different in the correctly specified diffusion case.


Here we introduce some notations and conventions used throughout this paper.
We largely abbreviate ``$n$'' from the notation like $t_{j}=t^{n}_{j}$ and $h=h_{n}$. 
For any vector variable $x=(x^{(i)})$, $\frac{\p}{\p x^{(i)}}$ stands for the parital derivative with respect to the $i$-th component of $x$ and we write $\p_x=\left(\frac{\p}{\p x^{(i)}}\right)_i$.
$I_d$ and $O$ denote the $d$-dimensional identity matrix and zero matrix, respectively. 
$\top$ stands for the transpose operator, and $v^{\otimes2}:= vv^\top$ for any matrix $v$.
The convergences in probability and in distribution are denoted by $\cip$ and $\cil$, respectively. All limits appearing below are taken for $n\to\infty$ unless otherwise mentioned.
For two nonnegative real sequences $(a_n)$ and $(b_n)$, we write $a_n \lesssim b_n$ if there exists a positive constant $C$ and $N\in\mathbb{N}$ such that $a_n \leq Cb_n$ for any $n\geq N$.
For any process $Y$, $\D_j Y$ denotes the $j$-th increment $Y_{t_{j}}-Y_{t_{j-1}}$.
For any matrix-valued function $f$ on $\mathbb{R}\times\Theta$, we write $f_s(\theta)=f(X_s,\theta)$. 
The L\'{e}vy measure of $Z$ is written as $\nu_0(dz)$, and the associated compensated Poisson random measure is represented by $\tilde{N}(ds,dz)$.
$\mca$ and $\tilde{\mca}$ stand for the infinitesimal generator and extended generator of $X$, respectively. 

The rest of this paper is organized as follows. 
In Section \ref{Notations}, we provide a brief overview of the Gaussian quasi-likelihood estimation. 
We also introduce assumptions used throughout of this paper.
Section \ref{Bootstrap} is the main body of this paper: first we construct a adjustment term for uniformly dealing with the difference of the convergence rate of the scale parameter $\gam$, and after that we propose our bootstrap method, and show its theoretical property.
Section \ref{Numerical Experiment} presents the finite sample performance of our method.
All of their proofs are given in Section \ref{Proofs}.

\section{Gaussian quasi-likelihood estimation and Assumptions}\label{Notations}
Since the explicit form of the transition probability of $X$ cannot be obtained in general, the estimation based on the genuine likelihood function is impractical.
In this section, we briefly explain the Gaussian quasi-likelihood estimation for our model, and introduce assumptions for its asymptotic results and our main results.
Building on the discrete-time approximation of \eqref{Model}:
\begin{equation*}
    X_{t_{j}}\approx X_{t_{j-1}}+h_na_{t_{j-1}}(\alpha)+c_{t_{j-1}}(\gam)\D_j Z,
\end{equation*}
we consider the stepwise Gaussian quasi-likelihood (GQL) function defined as follows:
\begin{align*}
    &\mbbh_{1,n}(\gam)=-\frac{1}{2h_n}\sumj \left\{h_n\log c^2_{t_{j-1}}(\gam)+\frac{(\D_j X)^2}{c^2_{t_{j-1}}(\gam)}\right\},\\
    &\mbbh_{2,n}(\al,\gam)=-\frac{1}{2h_n}\sumj \frac{(\D_j X-h_na_{t_{j-1}}(\al))^2}{c^2_{t_{j-1}}(\gam)}.
\end{align*}
Based on this GQL function, we define Gaussian quasi maximum likelihood estimator (GQMLE) $\tes:=(\ges,\aes)$ by
\begin{equation*}
\ges\in\argmax_{\gam\in\bar{\Theta}_\gam} \mbbh_{1,n}(\gam), \ \aes\in\argmax_{\al\in\bar{\Theta}_\al} \mbbh_{2,n}(\al,\ges).
\end{equation*}

We define an optimal parameter $\theta^\star:=(\gam^\star,\al^\star)$ of $\theta$ by 
\begin{equation*}
\gam^\star\in\argmax_{\gam\in\bar{\Theta}_\gam}\mbbh_1(\gam),\quad \al^\star\in\argmax_{\al\in\bar{\Theta}_\al}\mbbh_2(\al),
\end{equation*}
where $\mbbh_1: \Theta_\gam\mapsto \mbbr$ and $\mbbh_{2}:\Theta_\al\mapsto\mbbr$ are defined as follows:
\begin{align}
&\mbbh_1(\gam)=-\frac{1}{2}\int_\mbbr \left(\log c^2(x,\gam)+\frac{C^2(x)}{c^2(x,\gam)}\right)\pi_0(dx),\label{KL1}\\
&\mbbh_2(\al)=-\frac{1}{2}\int_\mbbr c^{-2}(x,\gam^\star)(A(x)-a(x,\al))^2\pi_0(dx). \label{KL2}
\end{align}

Now we introduce the technical assumptions for our main results.
Some comments on each assumption will be given after the all assumptions are mentioned.
Recall that the parameter space $\Theta$ is supposed to be a bounded convex domain.
We assume the following identifiability condition for $\mbbh_1(\gam)$ and $\mbbh_2(\al)$:

\begin{Assumption}\label{Identifiability}
$\theta^\star\in\Theta$,
and there exist positive constants $\chi_\gam$ and $\chi_\al$ such that for all $(\gam,\al)\in\Theta$,
\begin{align}
&\label{idg}\mbbh_1(\gam)-\mbbh_1(\gam^\star)\leq-\chi_\gam|\gam-\gam^\star|^2,\\
&\label{ida}\mbbh_2(\al)-\mbbh_2(\al^\star)\leq-\chi_\al|\al-\al^\star|^2.
\end{align}
\end{Assumption}

In the rest of this paper, we sometimes omit the optimal value $\theta^\star$, for instance, the abbreviated symbols $f_{s}$ and $f_{t_{j-1}}$ are used instead of $f_{s}(\theta^\star)$, and $f_{t_{j-1}}(\theta^\star)$, respectively.

\begin{Assumption}\label{Smoothness}
\begin{enumerate}
\item The coefficients $A$ and $C$ are Lipschitz continuous and twice differentiable, and their first and second derivatives are of at most polynomial growth.
\item The drift coefficient $a(\cdot,\al^\star)$ and scale coefficient $c(\cdot,\gam^\star)$ are Lipschitz continuous, and $c(x,\gam)\neq0$ for every $(x,\gam)$.
\item For each $i \in \left\{0,1\right\}$ and $k \in \left\{0,\dots,5\right\}$, the following conditions hold:
\begin{itemize}
\item The coefficients $a$ and $c$ admit extension in $\mathcal{C}(\mathbb{R}\times\bar{\Theta})$ and have the partial derivatives $(\partial_x^i \partial_\alpha^k a, \partial_x^i \partial_\gamma^k c)$ possessing extension in $\mathcal{C}(\mathbb{R}\times\bar{\Theta})$.
\item There exists nonnegative constant $C_{(i,k)}$ satisfying
\begin{equation}\label{polynomial}
\sup_{(x,\alpha,\gamma) \in \mathbb{R} \times \bar{\Theta}_\alpha \times \bar{\Theta}_\gamma}\frac{1}{1+|x|^{C_{(i,k)}}}\left\{|\partial_x^i\partial_\alpha^ka(x,\alpha)|+|\partial_x^i\partial_\gamma^kc(x,\gamma)|+|c^{-1}(x,\gamma)|\right\}<\infty.
\end{equation}
\end{itemize}
\end{enumerate}
\end{Assumption}

Note that since we impose the extension condition in Assumption \ref{Smoothness}, $\mbby(\theta):=(\mbby_1(\gam),\mbby_2(\al))$ admit extension in $\mcc(\bar{\Theta})$ as well.

For a function $\rho:\mbbr\to\mbbr_+$ and a signed measure $m$ on a one-dimensional Borel space, we define
\begin{equation}
    \nn ||m||_\rho=\sup\{|m(\rho)|: \text{$f$ is $\mbbr$-valued, $m$-measurable, and satisfies $|f|\leq\rho$} \}.
\end{equation}
\begin{Assumption}\label{Stability}
\begin{enumerate}
\item
There exists a probability measure $\pi_0$ such that for every $q>0$, we can find constants $a>0$ and $C_q>0$ for which 
\begin{equation}\label{Ergodicity}
\sup_{t\in\mathbb{R}_{+}} \exp(at) ||P_t(x,\cdot)-\pi_0(\cdot)||_{h_q} \leq C_qh_q(x),
\end{equation}
for any $x\in\mathbb{R}$ where $h_q(x):=1+|x|^q$.
\item 
For any $q>0$,  we have
\begin{equation}
\sup_{t\geq0}E[|X_t|^q]<\infty. 
\end{equation}
\end{enumerate}
\end{Assumption}

\medskip
We introduce a $p\times p$-matrix $\Gam:=\begin{pmatrix}\Gam_\gam&O\\ \Gam_{\al\gam}&\Gam_\al\end{pmatrix}$ whose components are defined by:
\begin{align*}
&\Gam_\gam:=\int_\mbbr\frac{\p_\gam^{\otimes2}c(x,\gam^\star)c(x,\gam^\star)-(\p_\gam c(x,\gam^\star))^{\otimes2}}{c^4(x,\gam^\star)}(C^2(x)-c^2(x,\gam^\star))\pi_0(dx)\\
&\qquad -2\int_\mbbr \frac{(\p_\gam c(x,\gam^\star))^{\otimes2}}{c^4(x,\gam^\star)}C^2(x)\pi_0(dx),\\
&\Gam_{\al\gam}:=\int_\mbbr \p_\al a(x,\al^\star)\p_\gam^\top c^{-2}(x,\gam^\star)(a(x,\al^\star)-A(x))  \pi_0(dx),\\
&\Gam_\al:=-\int_\mbbr\frac{\p_\al^{\otimes2} a(x,\al^\star)}{c^2(x,\gam^\star)}(a(x,\al^\star)-A(x))\pi_0(dx)-\int_\mbbr\frac{(\p_\al a(x,\al^\star))^{\otimes2}}{c^2(x,\gam^\star)}\pi_0(dx).
\end{align*}
The matrix $\Gam$ is the probability limit of the Hessian matrix of the GQL function.

\begin{Assumption}\label{nd}
$-\Gam_\gam$, $-\Gam_\al$, and $-\Gam$ are positive definite.
\end{Assumption}

Next, we introduce the assumption on the driving noise.
The diffusion case corresponds with (1), and the pure-jump L\'{e}vy driven case does with (2).

\begin{Assumption}\label{Moments}
Either (1) or (2) is satisfied.
\begin{enumerate}
\item \textbf{Diffusion case}: $Z$ is a standard Wiener process (in this case, $Z$ is often written as $w$). Furthermore, there exist the solutions of the following Poisson equations:
\begin{align}
&\label{PE1}\mca f_1(x)=\frac{\p_\gam c(x,\gam^\star)}{c^3(x,\gam^\star)}(C^2(x)-c^2(x,\gam^\star)),\\
&\label{PE2}\mca f_2(x)=\frac{\p_\al a(x,\al^\star)}{c^2(x,\gam^\star)}(A(x)-a(x,\al^\star)),
\end{align}
and the solutions $f_1$ and $f_2$ are differentiable, and they and their derivatives are of at most polynomial growth.
\item \textbf{Pure-jump L\'{e}vy driven case}: $Z$ is a pure-jump L\'{e}vy process satisfying
\begin{itemize}
\item $E[Z_1]=0$, $Var[Z_1]=1$, and $E[|Z_1|^q]<\infty$ for all $q>0$.
\item The Blumenthal-Getoor index (BG-index) of $Z$ is smaller than 2, that is, 
\begin{equation*}
\beta:=\inf_\gam\left\{\gam\geq0: \int_{|z|\leq1}|z|^\gam\nu_0(dz)<\infty\right\}<2.
\end{equation*}
\end{itemize}
\end{enumerate}
\end{Assumption}

\medskip
We make comments on our assumptions below.
\begin{itemize}
    \item Under Assumption \ref{Smoothness} and Assumption \ref{Moments}, the existence and uniqueness of the strong solution of \eqref{Model} and its Markov and time-homogeneous property are guaranteed (cf. \cite[Section 6]{App09}).
    \item 
    The Poisson equations \eqref{PE1} and \eqref{PE2} will play important role in dealing with the misspecification bias (cf. \cite{UchYos11}). 
    A sufficient condition for the existence and regularity of $f_1$ and $f_2$ is given for example in \cite{ParVer01}.
    Moreover, in one-dimensional case, the explicit forms of $f_1$ and $f_2$ are presented in \cite[Remark 2.2]{UchYos11}.
    \item In the pure-jump L\'{e}vy driven case, a similar misspecification bias also exists.
    However, we cannot follow the same way as the diffusion case:
    the generator $\mca$ of $X$ is given by 
    \begin{equation}
        \nn\mca f(x) = A(x)\p_x f(x)+\int_\mbbr (f(x+C(x)z)-f(x)-\p_x f(x)C(x)z)\nu_0(dz),
    \end{equation}
    for a suitable function $f$. The integral with respect to $\nu_0$ makes it difficult to ensure the existence and regularity of the solutions of equations like $\mca f=g$ with some functions $f$ and $g$. Here, $g$ corresponds with the misspecification bias term. 
    Alternatively, as in \cite{Ueh19}, we invoke the theory of extended Poisson equations introduced by \cite{VerKul11} to deal with the misspecification bias. Its definition is as follows:
    \begin{Def}\cite[Definition 2.1]{VerKul11}
    We say that a measurable function $f:\mbbr\to\mbbr$ belongs to the domain of the extended generator $\tilde{\mca}$ of a c\`{a}dl\`{a}g homogeneous Feller Markov process $Y$ taking values in $\mbbr$ if there exists a measurable function $g:\mbbr\to\mbbr$ such that the process
    \begin{equation*}
    f(Y_t)-\int_0^t g(Y_s)ds, \qquad t\in\mbbr^+,
    \end{equation*}
    is well defined and is a local martingale with respect to the natural filtration of $Y$ and every measure $P_x(\cdot):=P(\cdot|Y_0=x),\ x\in\mbbr$.
    For such a pair $(f,g)$, we write $f\in\mbox{Dom}(\tilde{\mca})$ and $\tilde{\mca} f\overset{EPE}=g$.
    \end{Def}
    As for the L\'{e}vy driven SDE case, the Feller property holds under Assumption \ref{Smoothness} (cf. \cite[3.1.1 (ii)]{Mas07}) and we consider the following extended Poisson equations:
    \begin{align}
    \tilde{\mca} g_1(x)&=-\frac{\p_{\gam} c(x,\gam^\star)}{c^3(x,\gam^\star)}(c^2(x,\gam^\star)-C^2(x)), \label{epe1}\\
    \tilde{\mca} g_2(x)&=-\frac{\p_{\al} a(x,\al^\star)}{c^2(x,\gam^\star)}(A(x)-a(x,\al^\star)). \label{epe2}
    \end{align}
    \cite[Proposition 3.5]{Ueh19} shows the existence and uniqueness of $g_1$ and $g_2$ under Assumption \ref{Smoothness}. Assumption \ref{Stability} and Assumption \ref{Moments}-(2).
    Although the regularity of $g_1$ and $g_2$ is not obtained except for the limited case, its weighted H\"{o}lder continuity is also ensured under the same assumptions, and it is enough for our asymptotic result.
    For more details, see the disscussion in \cite[Section 3]{Ueh19}.
\end{itemize}

\begin{table}[t]
\begin{center}
\caption{GQL approach for ergodic diffusion models and ergodic L\'{e}vy driven SDE models }
\begin{tabular}{cccc}
\hline
&&&\\[-3.5mm]
Model & \multicolumn{2}{c}{Rates of convergence}& Ref.\\
& drift  & scale\\ \hline 
correctly specified diffusion & $\sqrt{T_n}$& $\sqrt{n}$ &\cite{Kes97}, \cite{UchYos12}   \\[1.5mm]
misspecified diffusion & $\sqrt{T_n}$& $\sqrt{T_n}$ & \cite{UchYos11} \\ [1.5mm]
correctly specified L\'{e}vy driven SDE& $\sqrt{T_n}$& $\sqrt{T_n}$ &\cite{Mas13-1}, \cite{MasUeh17-2} \\ [1.5mm]
misspecified L\'{e}vy driven SDE &  $\sqrt{T_n}$ &$\sqrt{T_n}$ & \cite{Ueh19}\\ \hline
\end{tabular}
\label{comp}
\end{center}
\end{table}

Building on these assumptions, we can derive the asymptotic normality of $\tes$.
For its technical details, we refer to the references presented in the next theorem.
\begin{Thm}\label{angqmle}
Under Assumptions \ref{Identifiability}-\ref{Moments}, we can deduce the consistency and asymptotic normality of the GQMLE $\tes$: 
$\tes\cip \theta^\star$, and 
\begin{equation*}
    A_n(\tes-\theta^\star)\cil N\left(0,\Gamma^{-1}\Sigma(\Gamma^{-1})^\top\right),
\end{equation*}
where $A_n$ denotes the rate matrix of $\tes$ (cf. Table \ref{comp}).
In each case, 
\begin{equation*}
    \Sigma:=\begin{pmatrix}\Sig_{\gam}&\Sig_{\al\gam}\\\Sig_{\al\gam}^\top&\Sig_\al\end{pmatrix}
\end{equation*}
is explicitly given below:
\begin{itemize}
\item Correctly specified diffusion case (\cite{Kes97},\cite{UchYos12}):
\begin{equation*}
\Sig=\begin{pmatrix}2\int_\mbbr \frac{(\p_\gam c(x,\gam^\star))^{\otimes2}}{c^2(x,\gam^\star)}\pi_0(dx) & O\\ O &\int_\mbbr\frac{(\p_\al a(x,\al^\star))^{\otimes2}}{c^2(x,\gam^\star)}\pi_0(dx) \end{pmatrix}.
\end{equation*}
\item Misspecified diffusion case (\cite{UchYos11}):
\begin{align*}
&\Sig_\gam =\int (\p_x f_1(x) C(x))^{\otimes 2}\pi_0(dx),\\ 
&\Sig_{\al\gam}=\int \left(\frac{\p_\al a(x,\al^\star)}{c^2(x,\gam^\star)}-\p_x f_2(x)\right)C^2(x)(\p_x f_1(x))^\top \pi_0(dx),\\
&\Sig_\al=\int \left[\left(\frac{\p_\al a(x,\al^\star)}{c^2(x,\gam^\star)}-\p_x f_2(x)\right)C(x)\right]^{\otimes 2}\pi_0(dx).
\end{align*}
\item Correctly specified pure-jump L\'{e}vy driven case (\cite{Mas13-1}, \cite{MasUeh17-2})
\begin{align*}
&\Sig_\gam=\int_\mbbr\left(\frac{\p_\gam c(x,\gam^\star)}{c(x,\gam^\star)}\right)^{\otimes2}\pi_0(dx)\int_\mbbr z^4\nu_0(dz),\\
&\Sig_{\al\gam}=-\int_\mbbr\left(\frac{\p_\gam c(x,\gam^\star)}{c(x,\gam^\star)}\right)\left(\frac{\p_\al a(x,\al^\star)}{c(x,\gam^\star)}\right)^\top\pi_0(dx)\int_\mbbr z^3\nu_0(dz),\\
&\Sig_\al=\int_\mbbr\left(\frac{\p_\al a(x,\al^\star)}{c(x,\gam^\star)}\right)^{\otimes2}\pi_0(dx).
\end{align*}
\item Misspecified pure-jump L\'{e}vy driven case (\cite{Ueh19})
\begin{align*}
&\Sig_\gam=\int_\mbbr\int_\mbbr\left(\frac{\p_\gam c(x,\gam^\star)}{c^3(x,\gam^\star)}C^2(x)z^2+g_1(x+C(x)z)-g_1(x)\right)^{\otimes2}\pi_0(dx)\nu_0(dz),\\
&\Sig_{\al\gam}=-\int_\mbbr\int_\mbbr\left(\frac{\p_\gam c(x,\gam^\star)}{c^3(x,\gam^\star)}C^2(x)z^2+g_1(x+C(x)z)-g_1(x)\right)\\
&\qquad \qquad \quad\left(\frac{\p_\al a(x,\al^\star)}{c^2(x,\gam^\star)}C(x)z+g_2(x+C(x)z)-g_2(x)\right)^\top\pi_0(dx)\nu_0(dz),\\
&\Sig_\al=\int_\mbbr\int_\mbbr\left(\frac{\p_\al a(x,\al^\star)}{c^2(x,\gam^\star)}C(x)z+g_2(x+C(x)z)-g_2(x)\right)^{\otimes2}\pi_0(dx)\nu_0(dz),
\end{align*}
where the functions $g_1$ and $g_2$ are the solution of \eqref{epe1} and \eqref{epe2}.
\end{itemize}
\end{Thm}


\begin{Rem}
    In the case where either of the coefficients is correctly specified, we can also derive a similar asymptotic result to the above. 
    It is worth noticing that when $Z$ is a standard Wiener process and the drift coefficient is misspecified, the convergence rate of $\ges$ is still $\sqrt{n}$ as in the correctly specified case (cf. \cite{UchYos11}) since the fluctuation of the drift part is dominated by the diffusion part in $L_p$-sense ($p\geq2$). 
    By using this estimate, we consider the stepwise estimation procedure.
\end{Rem}

\begin{Rem}
    From Assumption \ref{Identifiability} and Assumption \ref{Smoothness}, we have
    \begin{equation*}
        \p_\al\mbbh_2(\alpha^\star)=\int_\mbbr \p_\al a(x,\al^\star)c^{-2}(x,\gam^\star)(A(x)-a(x,\al))\pi_0(dx)=0.
    \end{equation*}
    Hence the off-diagonal part $\Gam_{\al\gam}$ of $\Gam$ becomes zero matrix for a 
    scale coefficient $c(x,\gam)$ being linear with respect to $\gam$,
    and it also does when the drift coefficient is correctly specified.
\end{Rem}

\begin{Rem}
    Regardless of whether the model is correctly specified or not, it is easy to construct a consistent estimator of $\Gam$.
    The $p\times p$ matrix $\hat{\Gam}_n$ defined by
    \begin{equation*}
        \hat{\Gam}_n=\begin{pmatrix}\frac{1}{n}\p_\gam^{\otimes 2} \mbbh_{1,n}(\ges)&O\\ \frac{1}{nh_n}\p_\gam\p_\al \mbbh_{2,n}(\aes,\ges)& \frac{1}{nh_n}\p_\al^{\otimes2} \mbbh_{2,n}(\aes,\ges)\end{pmatrix},
    \end{equation*}
    is one example and this matrix works both in the correctly specified and misspecified case.
\end{Rem}


\section{Main results}\label{Bootstrap}
\subsection{Adjustment of convergence rate}

From table \ref{comp}, the difference of the convergence rate can be seen with respect to the scale estimator $\ges$; more specifically, its convergence rate is $\sqrt{n}$ in the correctly specified case, and otherwise it is $\sqrt{T_n}$.
However, since no one can distinguish whether the statistical model is correctly specified or not, as a matter of course, we cannot identify $A_n$ in advance.
Therefore we need a constructible alternative of $A_n$ for uniformly dealing with the all cases below.
To satisfy the demand, we introduce the adjustment term
\begin{equation}
    \nn b_n:=b_{1,n}+b_{2,n},
\end{equation}
where $b_{1,n}$ and $b_{2,n}$ are defined as 
\begin{align}
&b_{1,n}=\frac{\sumj (\D_j X)^4}{\sumj (\D_j X)^2}, \nn\\
&b_{2,n}=\exp\Bigg(-\Bigg\{\left|\frac{1}{n}\sumj \left[\frac{(\D_j X)^4}{3h_n^2}-\frac{2(\D_j X)^2c^2_{t_{j-1}}(\ges)}{h_n}+c^4_{t_{j-1}}(\ges)\right]\right|\nn\\
& \qquad \qquad \qquad+\left|\frac{1}{n}\sumj \left[\frac{(\D_j X)^4}{3h_n^2}-\frac{2(\D_j X)^2c^2_{t_{j-1}}(\ges)}{h_n}+c^4_{t_{j-1}}(\ges)\right]\right|^{-1}\Bigg\}\Bigg). \nn
\end{align}
Obviously, it can be constructed only by the observed data.
The next proposition provides the asymptotic behavior of $b_n$.

\begin{Prop}\label{balance}
Suppose that Assumption \ref{Smoothness}, Assumption \ref{Stability}, and Assumption \ref{Moments} hold.
Then the adjustment term $b_n$ behaves as follows:
\begin{itemize}
\item In the correctly specified diffusion case,
\begin{equation}\label{bc1}
\frac{b_n}{3h_n}\cip \frac{\int_\mbbr c^4(x,\gam^\star)\pi_0(dx)}{\int_\mbbr c^2(x,\gam^\star)\pi_0(dx)}.
\end{equation}
\item In the misspecified diffusion case,
\begin{equation}\label{bc2}
b_{n}\cip \exp\left(-\left\{\int_\mbbr(C^2(x)-c^2(x,\gam^\star))^2\pi_0(dx)+\left[\int_\mbbr(C^2(x)-c^2(x,\gam^\star))^2\pi_0(dx)\right]^{-1}\right\}\right)\neq 0.
\end{equation}
\item In the pure-jump L\'{e}vy driven case, 
\begin{equation}\label{bc3}
b_{n}\cip \frac{\int_\mbbr c^4(x,\gam^\star)\pi_0(dx)\int_\mbbr z^4\nu_0(dz)}{\int_\mbbr c^2(x,\gam^\star)\pi_0(dx)}.
\end{equation}
\end{itemize}
\end{Prop}

Hereafter we write $b^\star$ as the (scaled) limit of $b_n$ given in Proposition \ref{balance}.
The importance of Proposition \ref{balance} is that the convergence rate of $b_n$ is $h_n$
only in the correctly specified diffusion case, that is, the convergence rate of the scale estimator is equivalent to $\sqrt{\frac{T_n}{b_n}}$ up to constant.
Thus the new matrix
\begin{equation}
\hat{A}_n:=
\begin{pmatrix}
\sqrt{\frac{T_n}{b_n}} I_{p_\gam}& O\\ O& \sqrt{T_n} I_{p_\al}\nn
\end{pmatrix},
\end{equation} 
is constructed by the observed data, and serves as a good alternative of $A_n$.
Let 
\begin{align*}
    &B^\star=\begin{pmatrix}\frac{1}{\sqrt{b^\star}}I_{p_\gam}&O\\ O&I_{p_\al} \end{pmatrix}.
\end{align*}
A simple application of Slutky's lemma with the asymptotic normality of $\tes$ gives the following corollary.

\begin{Cor}\label{balaan}
     Under Assumptions \ref{Identifiability}-\ref{Moments}, we have
     \begin{equation}
         \hat{A}_n\hat{\Gam}_n(\tes-\theta^\star)\cil N\left(0,B^\star\Sig B^\star\right).
     \end{equation}
\end{Cor}

\subsection{Bootstrap Gaussian quasi-maximum likelihood estimator}

From now on, we consider the approximation of the distribution of $\hat{A}_n\hat{\Gam}_n
(\tes-\theta^\star)$ instead of $A_n\hat{\Gam}_n
(\tes-\theta^\star)$ since we can avoid checking whether the model is misspecified or not and the driving noise is Wiener or not.

We divide the set $\{1,\dots,n\}$ into $k_n$-blocks $(B_{k_i})_{i=1}^{k_n}$ defined by: 
\begin{equation*}
B_{k_i}:=\left\{j\in\{1,\dots,n\}: (i-1)c_n+1\leq j\leq ic_n\right\},
\end{equation*}
where $c_n=\frac{n}{k_n}$, and here $k_n$ and $c_n$ are supposed to be a positive integer for simplicity.
With bootstrap weights $\{w_i\}_{i=1}^{k_n}$, we define the bootstrap Gaussian quasi-score function $\mbbg^{B}_{n}(\theta):=\left(\mbbg^{B}_{1,n}(\gam), \mbbg^{B}_{2,n}(\al)\right)^\top$ as: 
\begin{align*}
&\mbbg^{B}_{1,n}(\gam)=\sum_{i=1}^{k_n}w_i\sum_{j\in B_{k_i}}\frac{\p_\gam c_{t_{j-1}}(\gam)}{c^3_{t_{j-1}}(\gam)}\left[h_nc^2_{t_{j-1}}(\gam)-(\D_j X)^2\right],\\
&\mbbg^{B}_{2,n}(\al)=\sum_{i=1}^{k_n}w_i\sum_{j\in B_{k_i}}\frac{\p_\al a_{t_{j-1}}(\al)}{c^2_{t_{j-1}}(\ges)}\left[\D_j X-h_na_{t_{j-1}}(\al)\right].
\end{align*}
We define our bootstrap estimator $\tes^\textbf{B}:=(\ges^{\textbf{B}},\aes^{\textbf{B}})$ as the solution of 
\begin{equation*}
    \left|\mbbg_{1,n}^{\textbf{B}}(\ges^{\textbf{B}})\right|+\left|\mbbg_{2,n}^{\textbf{B}}(\aes^{\textbf{B}})\right|=0.
\end{equation*}
For the bootstrap weights and block size, we assume:
\begin{Assumption}\label{weights}
    There exists a positive $\del\in\left(\frac{1}{2},1\right)$ such that  $k_n=O\left(T_n^{\del}\right)$, and 
    the bootstrap weights $\{w_i\}_{i=1}^{k_n}$ are i.i.d. nonnegative random variables and independent of $X=(X_t)_{t\geq 0}$ with $E[w_i]=1$, $E[w_i^2]=1$, and $E[|w_i|^{2+\del'}]<\infty$, for some $\del'>0$.
\end{Assumption}

In the rest of this paper, $\pb$ stands for the probability of bootstrap random variables, conditional on $\mcf$.
Analogously, $\eb$ represents the expectation with respect to $\pb$.
More specifically, for any bootstrap quantity $U_n(\cdot,\omega)$ and measurable set $A$, 
\begin{equation}
    \nn \pb\left(U_n\in A\right)=\pb\left(U_n(\cdot,\omega)\in A| \mbbx\right),
\end{equation}
where $\omega\in\Omega$.
Regarding $\pb$, $r_{nB}$ denotes a generic random vector fulfilling 
\begin{equation*}
\pb(|r_{nB}|>M)=o_p(1),
\end{equation*}
for any $M>0$.
Its explicit form depends on each context.

\begin{Rem}\label{third}
    For such a weighted bootstrap procedure, several papers often assume the additional condition $E[w_i^3]=1$ in order to fit the first three moments of the bootstrap distribution.
    A popular candidate of the distribution of $w_i$ is proposed by \cite{Mam93}:
    \begin{equation*}
        w_i=\begin{cases}\frac{1-\sqrt{5}}{2}& \text{with probability} \ p=\frac{\sqrt{5}+1}{2\sqrt{5}}\\
        \frac{1+\sqrt{5}}{2} & \text{with probability} \ 1-p\end{cases}
    \end{equation*}
    We have the following other choice: let $\frac{w_i}{4}$ be the beta distribution whose density function $f_\frac{w_i}{4}$ is given by
    \begin{equation*}
        f_\frac{w_i}{4}(x)=\begin{cases}\frac{1}{B(\frac{1}{2},\frac{3}{2})}x^{\frac{1}{2}-1}(1-x)^{\frac{3}{2}-1} & 0\leq x\leq 1\\
        0& \text{otherwise}
        \end{cases}
    \end{equation*}
    where $B$ denotes the beta function.
    Then the distribution of $w_i$ is continuous and satisfies Assumption \ref{weights} and $E[w_i^3]=1$.
    In our case, the latter one often gives numerically good results.
    This may be because $k_n$ is not so large and the limit distribution of the normalized bootstrap estimator is continuous. 
\end{Rem}

\begin{Rem}
    The upper bound of $\del$ in Assumption \ref{weights} is for making the misspecification bias in the bootstrap distribution asymptotically negligible. 
    More specifically, the bias can be written as
    \begin{equation}
       \nn \frac{1}{\sqrt{T_n}}\sum_{i=1}^{k_n} (w_i-1)\sum_{j\in B_{k_i}} (f_{t_{j}}-f_{t_{j-1}})=\frac{1}{\sqrt{T_n}}\sum_{i=1}^{k_n} (w_i-1)(f_{ic_nh_n}-f_{[(i-1)c_n+1]h_n}), 
    \end{equation}
     where $f$ denotes a solution of the (extended) Poisson equations introduced in the previous section.
     Under our assumptions, the bias term is evaluated as $O_p(\sqrt{k_n}/\sqrt{T_n})$ (cf. the proof of Theorem \ref{se}), and thus $\del<1$ is essential.
\end{Rem}

\begin{Example}
    Suppose that $p_\al = p_\gam = 1$ and that the coefficients are written as $a(x,\al)=\al a(x)$ and $c(x,\gam)=\gam c(x)$ for some $\mbbr$-valued smooth functions $a(x)$ and $c(x)$ with $\gam\geq0$. Then, given $\mbbx$, the bootstrap estimator is calculated as
    \begin{align*}
        &\ges^{\textbf{B}}=\sqrt{\frac{\displaystyle k_n\sum_{i=1}^{k_n}w_i\sum_{j\in B_{k_i}}\frac{(\D_j X)^2}{c^2_{t_{j-1}}}}{\displaystyle T_n\sum_{i=1}^{k_n}w_i}}, \quad
        \aes^{\textbf{B}}=\frac{\displaystyle\sum_{i=1}^{k_n}w_i\sum_{j\in B_{k_i}}\frac{\D_j X}{c^2_{t_{j-1}}}a_{t_{j-1}}}{\displaystyle h_n\sum_{i=1}^{k_n}w_i\sum_{j\in B_{k_i}}\frac{a^2_{t_{j-1}}}{c^2_{t_{j-1}}}}.
    \end{align*}
\end{Example}

Let
\begin{align*}
    \hat{B}_n=
    \begin{pmatrix}
    \frac{1}{\sqrt{T_nb_n}}I_{p_\gam}& O\\ O&\frac{1}{\sqrt{T_n}}I_{p_\al}
    \end{pmatrix},\ \ \bar{\Gam}_n=\begin{pmatrix}\frac{1}{n}\p_\gam^{\otimes2}\mbbh_{1,n}(\ges)&O\\ O&\frac{1}{nh_n}\p_\al^{\otimes2}\mbbh_{2,n}(\aes,\ges)\end{pmatrix}.
\end{align*}
For each $j\in\{1,\dots,n\}$, define the indicator function $\chi_j(s)$ by
\begin{align}\nn
     \chi_j(s)=
     \begin{cases}
     1, & s\in [t_{j-1},t_j),\\
     0, & \text{otherwise}.
     \end{cases}
\end{align}
The following theorem ensures the existence of the bootstrap estimator and the bootstrap consistency of our method.

\begin{Thm}\label{se}
Under Assumptions \ref{Identifiability}-\ref{Moments} and Assumption \ref{weights}, we have
\begin{equation}
    \pb\left(\tes^\textbf{B}\in\Theta \right)=1-o_p(1),
\end{equation}
and $\tes^{\textbf{B}}$ admits the following stochastic expansion:
\begin{align}
    &\hat{A}_n\bar{\Gam}_n(\hat{\theta}_{n}^{\textbf{B}}-\tes)=\hat{B}_n\sum_{i=1}^{k_n} (w_i-1)\sum_{j\in B_{k_i}}\begin{pmatrix}\displaystyle\frac{\p_\gam c_{t_{j-1}}(\ges)}{c^3_{t_{j-1}}(\ges)}\left[h_nc^2_{t_{j-1}}(\ges)-(\D_j X)^2\right] \\\displaystyle\frac{\p_\al a_{t_{j-1}}(\aes)}{c^2_{t_{j-1}}(\ges)}\left(\D_j X-h_na_{t_{j-1}}(\aes)\right)
\end{pmatrix}+r_{nB}.\label{se2}
\end{align}
Furthermore, the first term of the right-hand-side of \eqref{se2} can be expressed as:
\begin{itemize}
    \item In the correctly specified diffusion case, 
\begin{align}
    \nn &\hat{B}_n\sum_{i=1}^{k_n} (w_i-1)\sum_{j\in B_{k_i}}\begin{pmatrix}\displaystyle\frac{\p_\gam c_{t_{j-1}}(\ges)}{c^3_{t_{j-1}}(\ges)}\left\{h_nc^2_{t_{j-1}}(\ges)-(\D_j X)^2\right\} \\\displaystyle\frac{\p_\al a_{t_{j-1}}(\aes)}{c^2_{t_{j-1}}(\ges)}\left(\D_j X-h_na_{t_{j-1}}(\aes)\right)
    \end{pmatrix}\\
    & =\hat{B}_n\sum_{i=1}^{k_n} (w_i-1)\sum_{j\in B_{k_i}}\int_{(i-1)c_nh_n}^{ic_nh_n}\begin{pmatrix}\displaystyle2\frac{\p_\gam c_{t_{j-1}}}{c_{t_{j-1}}}(w_s-w_{t_{j-1}})\\\displaystyle \frac{\p_\al a_{t_{j-1}}}{c_{t_{j-1}}} \end{pmatrix}\chi_j(s)dw_s+r_{nB}\label{secsd}.
\end{align}
    \item In the missepecified diffusion case,
\begin{align}
    \nn &\hat{B}_n\sum_{i=1}^{k_n} (w_i-1)\sum_{j\in B_{k_i}}\begin{pmatrix}\displaystyle\frac{\p_\gam c_{t_{j-1}}(\ges)}{c^3_{t_{j-1}}(\ges)}\left\{h_nc^2_{t_{j-1}}(\ges)-(\D_j X)^2\right\} \\\displaystyle\frac{\p_\al a_{t_{j-1}}(\aes)}{c^2_{t_{j-1}}(\ges)}\left(\D_j X-h_na_{t_{j-1}}(\aes)\right)
    \end{pmatrix}\\
    &=\hat{B}_n\sum_{i=1}^{k_n}(w_i-1)\sum_{j\in B_{k_i}} \int_{(i-1)c_nh_n}^{ic_nh_n}\begin{pmatrix}
    \p_x f_1(X_s)
    \\\displaystyle
    \frac{\p_\al a_s}{c^2_s}-\p_xf_2(X_{s})
    \end{pmatrix}C_s\chi_j(s)dw_s+r_{nB}\label{semsd}.
\end{align}
    \item In the pure-jump L\'{e}vy driven
    case, 
    \begin{align}
        \nn & \hat{B}_n\sum_{i=1}^{k_n} (w_i-1)\sum_{j\in B_{k_i}}\begin{pmatrix}\displaystyle\frac{\p_\gam c_{t_{j-1}}(\ges)}{c^3_{t_{j-1}}(\ges)}\left\{h_nc^2_{t_{j-1}}(\ges)-(\D_j X)^2\right\} \\\displaystyle\frac{\p_\al a_{t_{j-1}}(\aes)}{c^2_{t_{j-1}}(\ges)}\left(\D_j X-h_na_{t_{j-1}}(\aes)\right)
        \end{pmatrix}\\
        & = \hat{B}_n\sum_{i=1}^{k_n} (w_i-1)\int_{(i-1)c_nh_n}^{ic_nh_n}\int_\mbbr \begin{pmatrix}\displaystyle \frac{\p_\gam c_{s-}}{c_{s-}}z^2 \\\displaystyle\frac{\p_\al a_s}{c_{s-}}z\end{pmatrix} \tilde{N}(ds,dz)+r_{nB}\label{sepjl}.
    \end{align}
    \item In the misspecified pure-jump L\'{e}vy driven
    case, 
    \begin{align}
        \nn & \hat{B}_n\sum_{i=1}^{k_n} (w_i-1)\sum_{j\in B_{k_i}}\begin{pmatrix}\displaystyle\frac{\p_\gam c_{t_{j-1}}(\ges)}{c^3_{t_{j-1}}(\ges)}\left\{h_nc^2_{t_{j-1}}(\ges)-(\D_j X)^2\right\} \\\displaystyle\frac{\p_\al a_{t_{j-1}}(\aes)}{c^2_{t_{j-1}}(\ges)}\left(\D_j X-h_na_{t_{j-1}}(\aes)\right)
        \end{pmatrix}\\
        & = \hat{B}_n\sum_{i=1}^{k_n} (w_i-1)\int_{(i-1)c_nh_n}^{ic_nh_n}\int_\mbbr \begin{pmatrix}\displaystyle \frac{\p_\gam c_{s-}}{c^3_{s-}}C_{s-}^2z^2+g_1(X_{s-}+C_{s-}z)-g_1(X_{s-}) \\\displaystyle\frac{\p_\al a_s}{c^2_{s-}}C_{s-}z+g_2(X_{s-}+C_{s-}z)-g_2(X_{s-})\end{pmatrix} \tilde{N}(ds,dz)+r_{nB}\label{sempjl}.
    \end{align}
\end{itemize}

Furthermore, we get the following convergence for all cases:
\begin{equation}
    \label{cip}
     \sup_{x\in\mbbr^{p_\al+p_\gam}}\left|\pb\left(\hat{A}_n\bar{\Gam}_n\left(\tes^{\textbf{B}}-\tes\right)\leq x\right)-P\left(\hat{A}_n\hat{\Gam}_n\left(\tes-\theta^\star\right)\leq x\right)\right|\cip0.
\end{equation}
\end{Thm}

\begin{Rem}\label{bqsf}
The stochastic expansion shown in Theorem \ref{se} suggests that in order to obtain the bootstrap percentile and confidence intervals, it suffices to generate the bootstrapped quasi-score function
\begin{equation*}
\hat{B}_n\sum_{i=1}^{k_n} (w_{i,l}-1)\sum_{j\in B_{k_i}}\begin{pmatrix}\displaystyle\frac{\p_\gam c_{t_{j-1}}(\ges)}{c^3_{t_{j-1}}(\ges)}\left[h_nc^2_{t_{j-1}}(\ges)-(\D_j X)^2\right] \\\displaystyle\frac{\p_\al a_{t_{j-1}}(\aes)}{c^2_{t_{j-1}}(\ges)}\left(\D_j X-h_na_{t_{j-1}}(\aes)\right)
\end{pmatrix},
\end{equation*}
instead of $\hat{A}_n\bar{\Gam}_n(\hat{\theta}_{n,l}^{\textbf{B}}-\tes)$.
Importantly, its generation only require the optimization to get $\tes$ while calculating $\hat{A}_n\bar{\Gam}_n(\hat{\theta}_{n,l}^{\textbf{B}}-\tes)$ needs some optimization method such as quasi-Newton method for obtaining each $\hat{\theta}_{n,l}^{\textbf{B}}$, thus resulting much smaller computational effort.
\end{Rem}


\section{Numerical experiment}\label{Numerical}
We consider the following data-generating model and statistical model:
\begin{align}
    \label{dgp} &dX_t=-\frac{1}{2}X_tdt+dZ_t,\quad X_0=0,\\
    \nn &dX_t=-\frac{1}{2}X_tdt+\frac{\gam}{\sqrt{1+X_t^2}}dZ_t, \quad \gam>0.
\end{align}
As for the distribution of $Z_1$, we consider the two cases $(i) \mcl(Z_1)=N(0,1)$ and $(ii)\mcl(Z_1)=bgamma(1,\sqrt{2},1,\sqrt{2})$, where $bgamma(\del_1,\gam_1,\del_2,\gam_2)$ is defined as the law of $\tau_1-\tau_2$ where for each $i\in\{1,2\}$, $\tau_i$ stands for a gamma random variable whose L\'{e}vy density is
\begin{align*}
f_{\tau^i}(z)=\frac{\del_i}{z}e^{-\gam_iz}, \quad z>0.
\end{align*}
For two pairs of the sample size and Terminal time $(n,T_n)=(5\times10^4, 200), (10^5, 500)$, we independently generate 1000 paths of \eqref{dgp} based on Euler-Maruyama scheme.
Concerning the bootstrap weights, we choose the beta distribution based random variables given in Remark \ref{third}.
We set $k_n$ as $25$ and $50$ for each pair.
Then we have $\log_{500}25\approx 0.518, ~ \log_{200}25\approx0.608, ~ \log_{500}50\approx 0.629$ and $\log_{200}50\approx0.738$.
Hence, in this setting, all of Assumptions \ref{Identifiability}-\ref{Moments} and Assumption \ref{weights} hold with the optimal value $\gam^\star=\sqrt{2}$ (cf. \cite{UchYos11} and \cite{Ueh19}).
Table \ref{nres} shows the actual coverage rate of all the 99\% bootstrap interval constructed by 1000 bootstrap replication.
Each coverage rate is approaching as $n$ and $T_n$ increases.
Compared with case (i), the coverage rate of case (ii) is slightly worse; this is probably because of the high kurtosis of the bilateral gamma distribution which makes the asymptotic variance of $\ges$ large.

\begin{table}[t]
\label{nres}
\caption{Coverage rate of 99\% bootstrap confidence interval}
\begin{tabular}{lllllll}
\hline
$n$ & $T_n$ & $k_n$ &  coverage rate (i) & coverage rate (ii) \\
\hline\hline
$10^5$ & $500$ & 25  &  0.962 & 0.952\\
$10^5$ & $500$ & 50  &  0.969 & 0.944\\
$5\times10^4$ & $200$  & 25  &  0.935 & 0.924\\
$5\times10^4$ & $200$  & 50  &  0.939 & 0.907\\
\hline
\end{tabular}
\end{table}



\section{Proofs}\label{Proofs}
Hereafter, $R(x)$ denotes a generic function being of at most polynomial growth.
Its form may vary depending on context.

\subsection{Auxiliary results for blocked sums}
We first prepare some lemmas repeatedly used in the proof of our main results.

\begin{Lem}\label{bsum} Suppose that Assumptions \ref{Smoothness}-\ref{Stability}, and Assumption \ref{Moments} hold. Then we have
    \begin{align}
        \label{bsum1} &\sum_{i=1}^{k_n}\left|\sum_{j\in B_{k_i}} R_{t_{j-1}}\intj (A_s-A_{t_{j-1}})ds\right|^2=O_p\left(\frac{n^2h_n^3}{k_n}\right),\\
        \label{bsum2} &\sum_{i=1}^{k_n}\left|\sum_{j\in B_{k_i}} R_{t_{j-1}}\left(\intj A_sds\right)^2\right|^2=O_p\left(\frac{n^2h_n^4}{k_n}\right),\\
        \label{bsum3} &\sum_{i=1}^{k_n}\left|\sum_{j\in B_{k_i}} R_{t_{j-1}}\intj C_{s-}dZ_s\right|^2=O_p(T_n).
    \end{align}
\end{Lem}
\begin{proof}
    \noindent \textbf{Proof of \eqref{bsum1}}:
    From the Lipschitz continuity of $A$, Jensen's inequality, and \cite[Lemma 5.3]{Mas13-1}, we have
    \begin{align}
        \nn & E\left[\sum_{i=1}^{k_n}\left|\sum_{j\in B_{k_i}} R_{t_{j-1}}\intj (A_s-A_{t_{j-1}})ds\right|^2\right]\\
        \nn & \lesssim \frac{nh_n}{k_n}\sum_{i=1}^{k_n}\sum_{j\in B_{k_i}} E\left[\left|R_{t_{j-1}}\right|^2\left[\intj E^{j-1}\left[|X_s-X_{t_{j-1}}|^2\right]ds\right]\right]\\
        \nn & \lesssim \frac{n^2h_n^3}{k_n}.
    \end{align}
    
    \medskip
    \noindent \textbf{Proof of \eqref{bsum2}}: It is straightforward from Jensen's inequality.
    
    \medskip
    \noindent \textbf{Proof of \eqref{bsum3}}: By applying Burkholder's inequality twice, it follows that
    \begin{align}
        \nn & E\left[\sum_{i=1}^{k_n}\left|\sum_{j\in B_{k_i}} R_{t_{j-1}}\intj C_{s-}dZ_s\right|^2\right]\\
        \nn & \lesssim \sum_{i=1}^{k_n} \sum_{j\in B_{k_i}}E\left[\left| R_{t_{j-1}}\right|^2E^{j-1}\left[\left|\intj C_{s-}dZ_s\right|^2\right]\right]\\
        \nn & \lesssim \sum_{i=1}^{k_n} \sum_{j\in B_{k_i}}E\left[\left| R_{t_{j-1}}\right|^2\intj E^{j-1}\left[\left|C_{s}\right|^2\right]ds\right]\\
        \nn & \lesssim T_n.
    \end{align}
\end{proof}

\begin{Lem}\label{wlbsum} Suppose that Assumptions \ref{Smoothness}-\ref{Stability} hold. Then under Assumption \ref{Moments}-(1), we have
\begin{align}
    \label{wbsum4} &\sum_{i=1}^{k_n}\left|\sum_{j\in B_{k_i}} R_{t_{j-1}}\left[\left(\intj C_{s}dw_s\right)^2-h_nC^2_{t_{j-1}}\right]\right|^2=O_p\left(nh_n^2\right),
\end{align}
and under Assumption \ref{Moments}-(2), 
\begin{align}
    \label{lbsum4} &\sum_{i=1}^{k_n}\left|\sum_{j\in B_{k_i}} R_{t_{j-1}}\left[\left(\intj C_{s-}dZ_s\right)^2-h_nC^2_{t_{j-1}}\right]\right|^2=O_p\left(T_n\right).
\end{align}
\end{Lem}

\begin{proof}
     \noindent\textbf{Proof of \eqref{wbsum4}}: 
     First we rewrite $\left(\intj C_{s}dw_s\right)^2-h_nC_{t_{j-1}}^2$ as 
     \begin{align}
         \nn& \left(\intj C_{s}dw_s\right)^2-h_nC_{t_{j-1}}^2\\
         \nn& =\left(\intj C_{s}dw_s\right)^2-E^{j-1}\left[\left(\intj C_{s}dw_s\right)^2\right]+E^{j-1}\left[\left(\intj C_{s}dw_s\right)^2\right]-h_nC_{t_{j-1}}^2.
     \end{align}
     From Assumption \ref{Stability} and Burkholder's inequality, there exists a positive constant $K$ such that
     \begin{align}
         \nn & E\left[\sum_{i=1}^{k_n}\left|\sum_{j\in B_{k_i}} R_{t_{j-1}}\left\{\left(\intj C_{s}dw_s\right)^2-E^{j-1}\left[\left(\intj C_{s}dw_s\right)^2\right]\right\}\right|^2\right]\\
         \nn & \lesssim h_n\sum_{i=1}^{k_n} \sum_{j\in B_{k_i}} E\left[\left|R_{t_{j-1}}\right|^2\intj E^{j-1}\left[1+|X_s-X_{t_{j-1}}|^K+|X_{t_{j-1}}|^K\right]ds\right]\\
         \nn & \lesssim nh_n^2.
     \end{align}
    Since It\^{o}'s formula leads to
    \begin{align}
        \nn  E^{j-1}\left[\left(\intj C_{s}dw_s\right)^2\right]&=\intj E^{j-1}\left[C_{s}^2\right]ds\\
        \nn &=h_nC_{t_{j-1}}^2+\intj \left(\int_{t_{j-1}}^s E^{j-1}\left[\mca C_u^2\right]du\right)ds,
    \end{align}
    in a similar manner to the proof of \eqref{bsum1}, we can obtain
    \begin{align}
        \nn & E\left[\sum_{i=1}^{k_n}\left|\sum_{j\in B_{k_i}} R_{t_{j-1}}\left\{E^{j-1}\left[\left(\intj C_{s}dw_s\right)^2\right]-h_nC_{t_{j-1}}^2\right\}\right|^2\right]\lesssim \frac{n^2h_n^4}{k_n}=o\left(nh_n^2\right),
    \end{align}
    and thus \eqref{wbsum4}.
    
    \medskip
    \noindent\textbf{Proof of \eqref{lbsum4}}: From Assumption \ref{Moments}, $\int_\mbbr z^2\nu_0(dz)=1$. Hence, It\^{o}'s formula yields that
     \begin{align}
         \nn &\left(\intj C_{s-}dZ_s\right)^2-h_nC_{t_{j-1}}^2\\
         \nn &=\intj \left(C_{s}^2-C_{t_{j-1}}^2\right)ds+2\intj \left(\int_{t_{j-1}}^s C_{u-}dZ_u\right)C_{s-}dZ_s+\intj\int_\mbbr C_{s-}^2 z^2\tilde{N}(ds,dz).
     \end{align}
     By taking a similar route to the proof of \eqref{wbsum4}, we can easily observe that
     \begin{equation}
         \nn E\left[\sum_{i=1}^{k_n}\left|\sum_{j\in B_{k_i}} R_{t_{j-1}}\intj \left(C_{s}^2-C_{t_{j-1}}^2\right)ds\right|^2\right]\lesssim \frac{n^2h_n^{5/2}}{k_n}.
     \end{equation}
     Since the second and third terms of the right-hand-side are martingale, the desired result directly follows from Burkholder's inequality and 
     \begin{equation}
         \nn \frac{n^2h_n^{5/2}}{k_nT_n}=\frac{nh_n^{3/2}}{k_n}\lesssim \sqrt{nh_n^2}=o(1).
     \end{equation}
\end{proof}

We will say that a matrix-valued function $f$ on $\mbbr$ is centered if $\pi_0(f)=0$ in the rest of this section.
\begin{Lem}\label{misse}
Suppose that a centered matrix-valued function $f$ is differentiable, and that it and its derivative are of at most polynomial growth. Then under Assumption \ref{Smoothness}, Assumption \ref{Stability}, and Assumption \ref{Moments}, we have
\begin{align}
    \sum_{i=1}^{k_n}\left|\sum_{j\in B_{k_i}}f_{t_{j-1}}\right|^2=O_p\left(\sqrt{nk_n}\right).
\end{align}
\end{Lem}

\begin{proof}
Applying \cite[Lemma 4]{Yos11}, we have
\begin{equation}
    \nn E\left[\left|\sum_{j\in B_{k_i}}\left(f_{t_{j-1}}-E[f_{t_{j-1}}]\right)\right|^2\right]\lesssim \sqrt{\frac{n}{k_n}}.
\end{equation}
From Assumption \ref{Stability}, we also have
\begin{align}
    \nn \left|\sum_{j\in B_{k_i}}E[f_{t_{j-1}}]\right|^2 &=\left|\sum_{j\in B_{k_i}}\int_\mbbr\int_\mbbr f(y) \left[P_{t_{j-1}}(x,dy)-\pi_0(dy)\right]\eta(dx)\right|^2\\
    \nn & \lesssim e^{-2a[(i-1)c_n+1]h_n}\left|\sum_{j=0}^{c_n-1} e^{-ajh_n}\right|^2\\
    \nn&\lesssim 1.
\end{align}
Thus we get the desired result.

\end{proof}

The following lemma verifies the probability limit of the sum of squared Wiener integrals.

\begin{Lem}\label{Wl2c}
Suppose that Assumption \ref{Smoothness}, Assumption \ref{Stability}, and Assumption \ref{Moments}-(1) hold.
We further assume that a differentiable function $f$ on $\mbbr$ and its derivative are of at most polynomial growth.
Then we have
\begin{align}
    &\frac{1}{nh_n^2}\sum_{i=1}^{k_n}\left\{\int_{(i-1)c_nh_n}^{ic_nh_n}\left[\sum_{j=(i-1)c_n+1}^{ic_n}f_{t_{j-1}}(w_s-w_{t_{j-1}})\chi_j(s)\right]dw_s\right\}^2\cip\frac{1}{2}\int_\mbbr \left(f(x)\right)^2\pi_0(dx), \label{wc1}\\
    &\frac{1}{T_n}\sum_{i=1}^{k_n}\left(\int_{(i-1)c_nh_n}^{ic_nh_n}\sum_{j=(i-1)c_n+1}^{ic_n}f_{s}\chi_j(s)dw_s\right)^2\cip \int_\mbbr \left(f(x)\right)^2\pi_0(dx).\label{wc2}
\end{align}
\end{Lem}

\begin{proof}
For simplicity, we write 
\begin{equation}
    Y_{i,t}=\int_{(i-1)c_nh_n}^{t}\left[\sum_{j=(i-1)c_n+1}^{ic_n}f_{t_{j-1}}(w_s-w_{t_{j-1}})\chi_j(s)\right]dw_s,  \quad t\in((i-1)c_nh_n, ic_nh_n].\nn
\end{equation}
From It\^{o}'s formula, it follows that
\begin{align}\nn
    \left(Y_{i,ic_nh_n}\right)^2=2\int_{(i-1)c_nh_n}^{ic_nh_n} Y_{i,s}dY_{i,s}+\sum_{j=(i-1)c_n+1}^{ic_n}\left[f^2_{t_{j-1}} \int_{(j-1)h_n}^{jh_n} (w_s-w_{t_{j-1}})^2ds\right].
\end{align}
Hence the left-hand-side of \eqref{wc1} can be rewritten as:
\begin{align}
    &\frac{1}{nh_n^2}\sum_{i=1}^{k_n}\left\{\int_{(i-1)c_nh_n}^{ic_nh_n}\left[\sum_{j=(i-1)c_n+1}^{ic_n}f_{t_{j-1}}(w_s-w_{t_{j-1}})\chi_j(s)\right]dw_s\right\}^2\nn\\
    &=\frac{2}{nh_n^2}\sum_{i=1}^{k_n}\int_{(i-1)c_nh_n}^{ic_nh_n}Y_{i,s}dY_{i,s}+\frac{1}{nh_n^2}\sumj \left[f^2_{j-1}\intj (w_s-w_{t_{j-1}})^2ds\right]\nn,
\end{align}
By utilizing Burkholder's inequality, we get
\begin{equation}
    \nn E\left[\left(\frac{1}{nh_n^2}\sum_{i=1}^{k_n}\int_{(i-1)c_nh_n}^{ic_nh_n}Y_{i,s}dY_{i,s}\right)^2\right]=O\left(\frac{1}{T_n}\right).
\end{equation}
Moreover, Fubini's theorem and Jensen's inequality lead to
\begin{align}
    \nn &E^{j-1}\left[\intj (w_s-w_{t_{j-1}})^2ds\right]=\frac{h_n^2}{2},\\
    \nn &E^{j-1}\left[\left\{\intj (w_s-w_{t_{j-1}})^2ds\right\}^2\right]\lesssim h_n^4.
\end{align}
Hence \cite[Lemma 9]{GenJac93} and the ergodic theorem imply \eqref{wc1}.
Next we show \eqref{wc2}.
First we decompose $\left(\int_{(i-1)c_nh_n}^{ic_nh_n}\sum_{j=(i-1)c_n+1}^{ic_n}f_{s}\chi_j(s)dw_s\right)^2$ as:
\begin{align*}
    &\left(\int_{(i-1)c_nh_n}^{ic_nh_n}\sum_{j=(i-1)c_n+1}^{ic_n}f_{s}\chi_j(s)dw_s\right)^2\\
    &=\left(\int_{(i-1)c_nh_n}^{ic_nh_n}\sum_{j=(i-1)c_n+1}^{ic_n}(f_{s}-f_{t_{j-1}})\chi_j(s)dw_s+\sum_{j=(i-1)c_n+1}^{ic_n}f_{t_{j-1}}\D_j w\right)^2.
\end{align*}
From the isometry property, Cauchy-Schwarz inequality, and Burkholder's inequality, we have 
\begin{align*}
    &E\left[\left(\int_{(i-1)c_nh_n}^{ic_nh_n}\sum_{j=(i-1)c_n+1}^{ic_n}(f_{s}-f_{t_{j-1}})\chi_j(s)dw_s\right)^2\right]\\
    &=\sum_{j=(i-1)c_n+1}^{ic_n}\intj E\left[(f_{s}-f_{t_{j-1}})^2\right]ds\\
    &\lesssim \sum_{j=(i-1)c_n+1}^{ic_n}\intj \sqrt{E\left[|X_s-X_{t_{j-1}}|^4\right]}ds\\
    &\lesssim c_nh_n^2.
\end{align*}
By the independence of the increments and the ergodic theorem, we also obtain
\begin{equation*}
    E\left[\left(\sum_{j=(i-1)c_n+1}^{ic_n}f_{t_{j-1}}\D_j w\right)^2\right]=h_n\sum_{j=(i-1)c_n+1}^{ic_n}E[f^2_{t_{j-1}}]\lesssim c_nh_n.
\end{equation*}
Hence Cauchy-Schwarz inequality yields that
\begin{align*}
    &\frac{1}{T_n}\sum_{i=1}^{k_n}\left(\int_{(i-1)c_nh_n}^{ic_nh_n}\sum_{j=(i-1)c_n+1}^{ic_n}f_{s}\chi_j(s)dw_s\right)^2\\
    &=\frac{1}{T_n}\sum_{i=1}^{k_n}\left(\sum_{j=(i-1)c_n+1}^{ic_n}f_{t_{j-1}}\D_j w\right)^2+O_p\left(\sqrt{h_n}\right)\\
    &=\frac{1}{T_n}\sum_{i=1}^{k_n}\left(\int_{(i-1)c_nh_n}^{ic_nh_n}\sum_{j=(i-1)c_n+1}^{ic_n}f_{t_{j-1}}\chi_j(s)dw_s \right)^2+O_p\left(\sqrt{h_n}\right).
\end{align*}
Then, by mimicking the proof of \eqref{wc1}, we can easily get \eqref{wc2}.
\end{proof}

\medskip

We next show a similar convergence result to Lemma \ref{Wl2c} when the driving noise is a pure-jump L\'{e}vy process.

\begin{Lem}\label{pLl2c}
Suppose that Assumptions \ref{Smoothness}-\ref{Stability}, and Assumption \ref{Moments}-(2) hold. Moreover, for two functions $f_1$ and $f_2$ on $\mbbr$, we assume the following conditions:
\begin{enumerate}
\item $f_1$ is differentiable, and it and its derivative is of at most polynomial growth.
\item There exists a positive constant $K$ such that for any $p\in(1,\infty)$ and $q=\frac{p}{p-1}$,
\begin{equation*}
\sup_{x,y\in\mbbr, x\neq y} \frac{\left|f_2(x)-f_2(y)\right|}{|x-y|^{1/p}(1+|x|^{qK}+|y|^{qK})}<\infty.
\end{equation*}
\end{enumerate}
Let $f(x,z)=f_1(x)z^\del+f_2(x+z)-f_2(x)$ for a fixed $\del\geq1$.
Then, we have
\begin{equation*}
\frac{1}{T_n}\sum_{i=1}^{k_n}\left(\int_{(i-1)c_nh_n}^{ic_nh_n}\int f(X_{s-},z)\tilde{N}(ds,dz)\right)^2 \cip \int_\mbbr\int_\mbbr \left(f(x,z)\right)^2\pi_0(dx)\nu_0(dz).
\end{equation*}
\end{Lem}

\begin{proof}
First we remark that since for all $0\leq s\leq t$,
\begin{equation*}
\int_s^t\int_\mbbr E\left[\left(f(X_{u-},z)\right)^2\right]\nu_0(dz) du<\infty,
\end{equation*}
the stochastic integral $\int_t^s\int f(X_{s-},z)\tilde{N}(ds,dz)$ is well defined (cf. \cite{App09}).
We take a similar way to the proof of the previous lemma.
Let 
\begin{equation*}
Y_{i,t}=\int_{(i-1)c_nh_n}^{t}\int f(X_{s-},z)\tilde{N}(ds,dz), \quad t\in((i-1)c_nh_n, ic_nh_n].
\end{equation*}
By applying It\^{o}'s formula, we have
\begin{align*}
\left(Y_{i,ic_nh_n}\right)^2&=2\int_{(i-1)c_nh_n}^{ic_nh_n} Y_{i,s-}dY_{i,s}+\int_{(i-1)c_nh_n}^{ic_nh_n}\int_\mbbr \left(f(X_{s-},z)\right)^2\tilde{N}(ds,dz)\\
&+\int_{(i-1)c_nh_n}^{ic_nh_n}\int_\mbbr \left(f(X_{s},z)\right)^2\nu_0(dz)ds.
\end{align*}
Hence it follows that 
\begin{align*}
&\frac{1}{T_n}\sum_{i=1}^{k_n}\left(\int_{(i-1)c_nh_n}^{ic_nh_n}\int f(X_{s-},z)\tilde{N}(ds,dz)\right)^2\\
&=\frac{1}{T_n}\left(2\sum_{i=1}^{k_n}\int_{(i-1)c_nh_n}^{ic_nh_n}Y_{i,s}dY_{i,s}+\int_{0}^{T_n}\int_\mbbr \left(f(X_{s-},z)\right)^2\tilde{N}(ds,dz)+\int_{0}^{T_n}\int_\mbbr \left(f(X_{s},z)\right)^2\nu_0(dz)ds\right).
\end{align*}
From the isometry property, we can easily observe that
\begin{equation*}
E\left[\left\{\frac{1}{T_n}\left(\sum_{i=1}^{k_n}\int_{(i-1)c_nh_n}^{ic_nh_n}Y_{i,s}dY_{i,s}+\int_{0}^{T_n}\int_\mbbr \left(f(X_{s-},z)\right)^2\tilde{N}(ds,dz)\right)\right\}^2\right]=O\left(\frac{1}{T_n}\right),
\end{equation*}
so that 
\begin{equation*}
\frac{1}{T_n}\sum_{i=1}^{k_n}\left(\int_{(i-1)c_nh_n}^{ic_nh_n}\int f(X_{s-},z)\tilde{N}(ds,dz)\right)^2=\frac{1}{T_n}\int_{0}^{T_n}\int_\mbbr \left(f(X_{s},z)\right)^2\nu_0(dz)ds+o_p(1).
\end{equation*}
Since we can pick a positive constant $\ep$ such that
\begin{align*}
\int_\mbbr \left(f(x,z)\right)^2\nu_0(dz)\lesssim \left(1+|x|^C\right)\int_\mbbr \left(|z|^{\beta+\ep}\vee |z|^{2\del}\right)\nu_0(dz)\lesssim 1+|x|^C,
\end{align*}
for all $x\in\mbbr$, the desired result follows from the ergodic theorem.

\end{proof}

\subsection{Proof of Proposition \ref{balance}}
For simplicity, we write
\begin{equation}
    \xi_j=\D_j X-C_{t_{j-1}}\D_j Z=\intj A_sds+\intj (C_{s-}-C_{t_{j-1}})dZ_s,\nn
\end{equation}
and we divide the proof into the diffusion case and pure-jump L\'{e}vy driven case.

\subsubsection{Diffusion case}
For any $q\geq 2$, from Burkholder's inequality and Assumption \ref{Smoothness}, we have
\begin{align}
     \label{bhd}&E[\left|\xi_j\right|^q]\lesssim h_n^{q}, \\
     \label{cbhd}&E^{j-1}[\left|\xi_j\right|^q]\lesssim h_n^{q}R_{t_{j-1}}.
\end{align}
Combined with H\"{o}lder's inequality and the ergodic theorem, it follows from \cite[Lemma 9]{GenJac93} that
\begin{align}
    \label{mom2}\frac{1}{T_n}\sumj (\D_j X)^2&=\frac{1}{T_n}\sumj C^2_{t_{j-1}}(\D_j Z)^2+O_p\left(\sqrt{h_n}\right)\cip \int C^2(x)\pi_0(dx),\\
    \label{mom4d}\frac{1}{nh_n^2}\sumj (\D_j X)^4&=\frac{1}{n}\sumj \frac{(\D_j Z)^4}{h_n^2}C^4_{t_{j-1}}+O_p\left(\sqrt{h_n}\right) \cip 3\int C^4(x)\pi_0(dx). 
\end{align}
Hence Slutsky's theorem leads to
\begin{align}
\nn\frac{b_{1,n}}{3h_n}\cip \frac{\int C^4(x)\pi_0(dx)}{\int C^2(x)\pi_0(dx)}.
\end{align}
Next we look at $b_{2,n}$.
By applying the Taylor's expansion, we have
\begin{align}
    &\frac{1}{n}\sumj \left[-2\frac{(\D_j X)^2}{h_n}c^2_{t_{j-1}}(\ges)+c_{t_{j-1}}^4(\ges)\right]\nn \\
    &=\frac{1}{n}\sumj \left[-2\frac{(\D_j X)^2}{h_n}c^2_{t_{j-1}}+c_{t_{j-1}}^4\right] \nn \\
    &+ \frac{1}{n}\int_0^1\sumj \left[-2\frac{(\D_j X)^2}{h_n}\p_\gam c^2_{t_{j-1}}(\ges+u(\gam^\star-\ges))+\p_\gam c_{t_{j-1}}^4(\ges+u(\gam^\star-\ges))\right]du[\ges-\gam^\star] \nn
\end{align}
Let $a_n$ be the convergence rate of $\ges$. 
Since $\p_\gam c^2$ and $\p_\gam c^4$ are of at most polynomial growth and $a_n(\ges-\gam^\star)=O_p(1)$, we can deduce that the second term of the right-hand-side is $O_p(\frac{1}{a_n})$.
Hence \eqref{bhd} and \eqref{cbhd} lead to
\begin{align}
    \nn &\frac{1}{n}\sumj \left[\frac{(\D_j X)^4}{3h_n^2}-2\frac{(\D_j X)^2}{h_n}c^2_{t_{j-1}}(\ges)+c_{t_{j-1}}^4(\ges)\right]\\
    &=\frac{1}{n}\sumj \left[\frac{(\D_j Z)^4}{3h_n^2}C^4_{t_{j-1}}-2\frac{(\D_j Z)^2}{h_n}c^2_{t_{j-1}}+c_{t_{j-1}}^4\right]+O_p\left(\sqrt{h_n}\vee\frac{1}{a_n}\right)\nn.
\end{align}
In the misspecified case, we can similarly observe that the first term of the right-hand-side converges to 
\begin{equation}
    \nn \int (C^2(x)-c^2(x))^2\pi_0(dx),
\end{equation}
in probability.
Then \eqref{bc2} follows from the continuous mapping theorem.
In the correctly specified case, we obtain
\begin{align}
    &E^{j-1}\left[\frac{(\D_j Z)^4}{3h_n^2}C^4_{t_{j-1}}-2\frac{(\D_j Z)^2}{h_n}c^2_{t_{j-1}}+c_{t_{j-1}}^4\right]=0\nn,\\
    &E^{j-1}\left[\left(\frac{(\D_j Z)^4}{3h_n^2}C^4_{t_{j-1}}-2\frac{(\D_j Z)^2}{h_n}c^2_{t_{j-1}}+c_{t_{j-1}}^4\right)^2\right]\lesssim R_{t_{j-1}}.\nn
\end{align}
\cite[Lemma 9]{GenJac93} yields that for a positive constant $\del\in(0,1/2)$, we have
\begin{equation}
    \frac{1}{n}\sumj \left[\frac{(\D_j Z)^4}{3h_n^2}C^4_{t_{j-1}}-2\frac{(\D_j Z)^2}{h_n}c^2_{t_{j-1}}+c_{t_{j-1}}^4\right]=o_p\left(n^{-\frac{1}{2}+\del}\right)\nn.
\end{equation}
Since $a_n=\sqrt{n}$ and $\sqrt{h_n}\vee\frac{1}{\sqrt{n}}=\sqrt{h_n}$, we arrive at
\begin{equation}
    \frac{1}{n}\sumj \left[\frac{(\D_j X)^4}{3h_n^2}-2\frac{(\D_j X)^2}{h_n}c^2_{t_{j-1}}(\ges)+c_{t_{j-1}}^4(\ges)\right]=o_p\left(n^{-\frac{1}{2}+\del}\wedge h_n^{\frac{1}{2}-\del}\right)=o_p\left(h_n^{\frac{1}{2}-\del}\right),
\end{equation}
and consequently,
\begin{align}
    \frac{b_{2,n}}{h_n}\lesssim \frac{1}{h_n}\exp\left(-\frac{h_n^{-\frac{1}{2}+\del}}{|o_p(1)|}\right)=o_p(1) \nn,
\end{align}
and this concludes \eqref{bc1}.

\subsubsection{Pure-jump L\'{e}vy driven case}
Since the route is quite similar to the diffusion case, we omit some details.
For any $q\geq2$, Burkholder's inequality and Assumption \ref{Smoothness} imply that
\begin{align}
     \label{bhl} &E[\left|\xi_j\right|^q]\lesssim h_n^2, \\
     \label{cbhl}&E^{j-1}[\left|\xi_j\right|^q]\lesssim h_n^2R_{t_{j-1}}.
\end{align}
Hence by making use of \cite[Lemma 9]{GenJac93}, we get
\begin{align}
    \label{mom2}&\frac{1}{T_n}\sumj (\D_j X)^2=\frac{1}{T_n}\sumj C^2_{t_{j-1}}(\D_j Z)^2+O_p\left(\sqrt{h_n}\right)\cip \int C^2(x)\pi_0(dx),\\
    \label{mom4l}&\frac{1}{nh_n}\sumj (\D_j X)^4=\frac{1}{n}\sumj C^4_{t_{j-1}}\frac{(\D_j Z)^4}{h_n}+O_p\left(\sqrt{h_n}\right)\cip \int C^4(x)\pi_0(dx)\int z^4\nu_0(dz),
\end{align}
and it is immediate from Slutsky's theorem that 
\begin{align}
\nn b_{1,n}\cip \frac{\int C^4(x)\pi_0(dx)\int z^4\nu_0(dz)}{\int C^2(x)\pi_0(dx)}.
\end{align}
From \eqref{mom4l} and a similar estimates to the diffusion case, we have
\begin{align}
&\nn\frac{1}{n}\sumj \left[\frac{(\D_j X)^4}{3h_n^2}-2\frac{(\D_j X)^2}{h_n}c^2_{t_{j-1}}(\ges)+c_{t_{j-1}}^4(\ges)\right]\to\infty.
\end{align}
Since the function $h(x)=\exp[-(|x|+1/|x|)]$ tends to 0 as $x\to\infty$, we obtain \eqref{bc3}.

\subsection{Proof of Theorem \ref{se}}
The essence of this proof stems from \cite{ChaBos05}.
For simplicity, we hereafter write
\begin{align}
    \nn\zeta_{j}(\gam)=\frac{\p_\gam c_{t_{j-1}}(\gam)}{c^3_{t_{j-1}}(\gam)}\left[h_nc^2_{t_{j-1}}(\gam)-(\D_j X)^2\right],\\
    \nn\eta_j(\al,\gam)=\frac{\p_\al a_{t_{j-1}}(\al)}{c^2_{t_{j-1}}(\gam)}\left[\D_j X-h_na_{t_{j-1}}(\al)\right].
\end{align}
Since the matrix $\bar{\Gam}_n$ is block diagonal, we divide the proof of \eqref{se2} into the scale part: 
\begin{align}
    \label{seg1}\sqrt{\frac{T_n}{b_n}}\frac{1}{T_n}\p_\gam\mbbg_{1,n}(\ges)(\bges-\ges)&=\frac{1}{\sqrt{T_nb_n}}\sum_{i=1}^{k_n}(w_i-1) \sum_{j\in B_{k_i}} \zeta_{j}(\ges)+r_{nB}\\
    \label{seg2}&=\frac{1}{\sqrt{T_nb_n}}\sum_{i=1}^{k_n}(w_i-1) \sum_{j\in B_{k_i}} \zeta_{j}(\gam^\star)+r_{nB},
\end{align}
and drift part:
\begin{align}
    \label{sea1}\frac{1}{\sqrt{T_n}}\p_\gam\mbbg_{2,n}(\aes)(\baes-\aes)&=\frac{1}{\sqrt{T_n}}\sum_{i=1}^{k_n}(w_i-1) \sum_{j\in B_{k_i}} \eta_{j}(\aes,\gam^\star)+r_{nB}\\
    \label{sea2}&=\frac{1}{\sqrt{T_n}}\sum_{i=1}^{k_n}(w_i-1) \sum_{j\in B_{k_i}} \eta_{j}(\al^\star,\gam^\star)+r_{nB}.
\end{align}

\medskip
\subsubsection{Scale part}

Define the function $\mbbh^B_n$ on $\mbbr^{p_\gam}$ by
\begin{align*}
\mbbh^B_n(t)&=\frac{1}{\sqrt{T_nb_n}}\sum_{i=1}^{k_n} w_i\sum_{j\in B_{j}}\left\{\zeta_j\left(\ges+\sqrt{\frac{b_n}{T_n}} t\right)-\zeta_j\left(\ges\right)\right\}-\frac{1}{T_n}\p_\gam\mbbg_{1,n}(\ges)t.
\end{align*}
First we show that
\begin{equation}\label{L2ignore}
\eb\left[\sup_{|t|\leq {K_n}}\left|\mbbh_n^{B}(t)\right|^2\right]=o_p(1),
\end{equation}
for any positive sequence $(K_n)$ fulfilling $K_n=o_p(\sqrt{k_n})$.
For instance, $K_n=T_n^{1/4}$ satisfies the above condition. 
Combined with Proposition \ref{balance} and the consistency of $\ges$, we have $\sqrt{\frac{b_n}{T_n}} K_n=o_p(1)$ and we may and do assume $\ges+\sqrt{\frac{b_n}{T_n}} t\in\Theta_\gam$ as long as $|t|\leq K_n$.
By applying Taylor's formula twice, we obtain
\begin{align*}
&\mbbh^B_n(t)=\mbbh^B_{1,n}(t)+\mbbh^B_{2,n}(t)+\mbbh^B_{3,n}(t),
\end{align*}
where for notational simplicity, we write
\begin{align*}
&\mbbh^B_{1,n}(t)=\frac{1}{T_n}\sum_{i=1}^{k_n}(w_i-1)\sum_{j\in B_{k_i}}  t^\top\p_\gam \zeta_j(\gam^\star),\\
&\mbbh^B_{2,n}(t)=\frac{1}{T_n}\sum_{i=1}^{k_n}(w_i-1)\sum_{j\in B_{k_i}} \int_0^1\p_\gam^{\otimes 2} \zeta_j(\ges+s(\gam^\star-\ges)) ds[t,\ges-\gam^\star],\\
&\mbbh^B_{3,n}(t)=\sqrt{\frac{b_n}{T_n^3}}\sum_{i=1}^{k_n}w_i \sum_{j\in B_{k_i}} \int_0^1 \p_\gam^{\otimes 2} \zeta_j\left(\ges+\sqrt{\frac{b_n}{T_n}}ut\right) du[t,t].
\end{align*}
From now on, we separately look at these three ingredients.
Assumption \ref{weights} yields that
\begin{equation*}
\eb\left[\sup_{|t|\leq {K_n}}\left|\mbbh^B_{1,n}(t)\right|^2\right]\leq \frac{K_n^2}{T_n^2}\sum_{i=1}^{k_n}\left|\sum_{j\in B_{k_i}} \p_\gam \zeta_j(\gam^\star)\right|^2.
\end{equation*}
Decompose $\p_\gam\zeta_j(\gam^\star)$ as $\p_\gam\zeta(\gam^\star)=\zeta_{1,j}+\zeta_{2,j}$ where
\begin{align*}
&\zeta_{1,j}=h_n\frac{c_{t_{j-1}}\p_\gam^{\otimes 2}c_{t_{j-1}}-(\p_\gam c_{t_{j-1}})^{\otimes 2}}{c^4_{t_{j-1}}}\left(c^2_{t_{j-1}}-C^2_{t_{j-1}}\right)+2h_n\frac{(\p_\gam c_{t_{j-1}})^{\otimes 2}}{c^4_{t_{j-1}}}C^2_{t_{j-1}},\\
&\zeta_{2,j}=\frac{c_{t_{j-1}}\p_\gam^{\otimes 2}c_{t_{j-1}}-3(\p_\gam c_{t_{j-1}})^{\otimes 2}}{c^4_{t_{j-1}}}\left[h_nC_{t_{j-1}}^2-(\D_j X)^2\right].
\end{align*}
Jensen's inequality and the ergodic theorem yield that 
\begin{equation}\label{l2xi1}
\frac{1}{T_n^2}\sum_{i=1}^{k_n}\left|\sum_{j\in B_{k_i}}  \zeta_{1,j}\right|^2\leq \frac{1}{nk_n}\sum_{i=1}^{k_n}\sum_{j\in B_{k_i}}  \zeta_{1,j}^2=O_p\left(k_n^{-1}\right).
\end{equation}
It follows from Lemma \ref{bsum} and Lemma \ref{wlbsum} that
\begin{align}
     \nn \frac{1}{T_n^2}\sum_{i=1}^{k_n}\left|\sum_{j\in B_{k_i}}  \zeta_{2,j}\right|^2 =
     \begin{cases}
     O_p\left(n^{-1}\right),& \text{in the diffusion case,}\\
     O_p\left(T_n^{-1}\right),& \text{in the pure-jump L\'{e}vy driven case.}
     \end{cases}
\end{align}
From these estimates, we get
\begin{equation}\label{l2mbh1}
    \eb\left[\sup_{|t|\leq K_n}\left|\mbbh^B_{1,n}(t)\right|^2\right]=O_p(K_n^2k_n^{-1})=o_p(1).
\end{equation}

\medskip

Assumption \ref{weights} leads to
\begin{align}
    \nn \eb\left[\sup_{|t|\leq K_n}\left|\mbbh^B_{2,n}(t)\right|^2\right]\leq \frac{K_n^2}{T_n^2}\left|\ges-\gam^\star\right|^2\sum_{i=1}^{k_n}\left|\sum_{j\in B_{k_i}} \int_0^1\p_\gam^{\otimes 2} \zeta_j(\ges+s(\gam^\star-\ges)) ds\right|^2.
\end{align}
From Assumption \ref{Smoothness}, there exists a positive constants $M_1$ and $M_2$ such that
\begin{align}
    &\nn \left|\int_0^1\p_\gam^{\otimes 2} \zeta_j(\ges+s(\gam^\star-\ges)) ds\right|\\
    &\nn \lesssim \left(1+|X_{t_{j-1}}|^{M_1}\right)\left\{h_n(1+|X_{t_{j-1}}|^{M_2})+\left[(\D_j X)^2-h_nC_{t_{j-1}}\right]\right\}.
\end{align}
Combined with Jensen's inequality, Lemma \ref{bsum}, and Lemma \ref{wlbsum}, we obtain
\begin{align}
    \label{poly} E\left[\left|\sum_{j\in B_{k_i}} \int_0^1\p_\gam^{\otimes 2} \zeta_j(\ges+s(\gam^\star-\ges)) ds\right|^2\right]
     \lesssim \frac{T_n^2}{k^2_n}.
\end{align}
Hence the tightness of $\sqrt{T_n}(\ges-\gam^\star)$ implies that 
\begin{equation}\label{l2mbh2}
    \eb\left[\sup_{|t|\leq K_n}\left|\mbbh^B_{2,n}(t)\right|^2\right]=O_p\left(K_n^2k_n^{-1}T_n^{-1}\right)=o_p(1).
\end{equation}

\medskip
We rewrite $\mbbh^B_{3,n}(t)$ as
\begin{align}
    \nn \mbbh^B_{3,n}(t)&=\sqrt{\frac{b_n}{T_n^3}}\sum_{i=1}^{k_n}(w_i-1) \sum_{j\in B_{k_i}} \int_0^1 \p_\gam^{\otimes 2} \zeta_j\left(\ges+\sqrt{\frac{b_n}{T_n}}ut\right) du[t,t]\\
    \nn &+\sqrt{\frac{b_n}{T_n^3}}\sumj \int_0^1 \p_\gam^{\otimes 2} \zeta_j\left(\ges+\sqrt{\frac{b_n}{T_n}}ut\right) du[t,t].
\end{align}
Recall that $b_n=O_p(1)$.
By utilizing Cauchy-Schwartz inequality and the estimates in \eqref{poly}, we have
\begin{align}
    \nn &\eb\left[\sup_{|t|\leq K_n}\left|\sqrt{\frac{b_n}{T_n^3}}\sum_{i=1}^{k_n}(w_i-1) \sum_{j\in B_{k_i}} \int_0^1 \p_\gam^{\otimes 2} \zeta_j\left(\ges+\sqrt{\frac{b_n}{T_n}}ut\right) du[t,t]\right|^2\right]\\
    \nn & \lesssim \frac{K_n^2b_n}{T_n^3}\sum_{i=1}^{k_n}\sup_{|t|\leq K_n}\left|\sum_{j\in B_{k_i}} \int_0^1 \p_\gam^{\otimes 2} \zeta_j\left(\ges+\sqrt{\frac{b_n}{T_n}}ut\right) du\right|^2\\
    \nn & =O_p(K_n^2k_n^{-1}T_n^{-1}). 
\end{align}
Similarly, for any $C>0$, we have 
\begin{equation}
    \nn\sup_{|t|\leq K_n}\left|\sqrt{\frac{b_n}{T_n^3}}\sumj \int_0^1 \p_\gam^{\otimes 2} \zeta_j\left(\ges+\sqrt{\frac{b_n}{T_n}}ut\right) du[t,t]\right|^2=O_p\left(K_n^2T_n^{-1}\right),
\end{equation}
so that 
\begin{equation}\label{l2mbh3}
    \eb\left[\sup_{|t|\leq K_n}\left|\mbbh^B_{3,n}(t)\right|^2\right]=O_p\left(K_n^2T_n^{-1}\right)=o_p(1).
\end{equation}
Putting \eqref{l2mbh1}, \eqref{l2mbh2}, and \eqref{l2mbh3} together, we arrive at \eqref{L2ignore}.

\medskip
Next we observe that
\begin{equation}
    \label{op1}\pb\left(\inf_{|t|=K_n}\left[-\frac{1}{\sqrt{T_nb_n}}\sum_{i=1}^{k_n} w_i\sum_{j\in B_{k_i}} t^\top\zeta_j\left(\ges+\sqrt{\frac{b_n}{T_n}}t\right)\right]>0\right)=o_p(1).
\end{equation}
By using the estimates in \cite{Kes97}, \cite{UchYos11}, \cite{Mas13-1}, and \cite{Ueh19}, it is easy to observe that
\begin{equation}\label{pconv}
\frac{1}{T_n}\p_\gam\mbbg_{1,n}(\ges) \cip -\Gam_\gam>0.
\end{equation}
We hereafter write the smallest eigenvalue of $-\Gam_\gam$ and $\frac{1}{T_n}\p_\gam\mbbg_{1,n}(\ges)$ as $\lam_\gam$ and $\lam_{\gam,n}$, respectively.
From \eqref{pconv}, we may and do assume that for a fixed $\del\in(0,\lam_\gam)$ and enough large $n$, $\frac{1}{T_n}\p_\gam\mbbg_{1,n}(\ges)$ is positive definite and $|\lam_\gam-\lam_{\gam,n}|<\del$ without loss of generality.
For such $n$, and the same sequence $(K_n)$ as the estimates of $\mbbh_n^B$, it follows that
\begin{align*}
&\pb\left(\inf_{|t|=K_n}\left[-\frac{1}{\sqrt{T_nb_n}}\sum_{i=1}^{k_n} w_i\sum_{j\in B_{k_i}} t^\top\zeta_j\left(\ges+\sqrt{\frac{b_n}{T_n}}t\right)\right]>0\right)\\
&\geq \pb\left(\inf_{|t|=K_n}\left\{-\frac{1}{\sqrt{T_nb_n}}\sum_{i=1}^{k_n} w_i\sum_{j\in B_{k_i}} t^\top\zeta_j\left(\ges\right)+t^\top\mbbh^B_n(t)+\frac{1}{T_n}\p_\gam\mbbg_{1,n}(\ges)[t,t]\right\}>0\right)\\
&\geq \pb\left(-\sup_{|t|=K_n}\frac{1}{\sqrt{T_nb_n}}\sum_{i=1}^{k_n} w_i\sum_{j\in B_{k_i}}t^\top\zeta_j\left(\ges\right)-\sup_{|t|=K_n}\left|t^\top\mbbh^B_n(t)\right|>-\inf_{|t|=K_n}\frac{1}{T_n}\p_\gam\mbbg_{1,n}(\ges)[t,t]\right)\\
&\geq 1-\pb\left(\left|\frac{1}{\sqrt{T_nb_n}}\sum_{i=1}^{k_n} w_i\sum_{j\in B_{k_i}}\zeta_j\left(\ges\right)\right|+\sup_{|t|=K_n}\left|\mbbh^B_n(t)\right|\geq (\lam_\gam-\del)K_n\right)\\
&\geq 1-\pb\left(\left|\frac{1}{\sqrt{T_nb_n}}\sum_{i=1}^{k_n} w_i\sum_{j\in B_{k_i}}\zeta_j\left(\ges\right)\right|\geq\frac{(\lam_\gam-\del)K_n}{2}\right)-\pb\left(\sup_{|t|=K_n}\left|\mbbh^B_n(t)\right|\geq\frac{(\lam_\gam-\del)K_n}{2}\right).
\end{align*}
For abbreviation, let 
\begin{equation}
    \nn M_n=M(K_n)=\frac{(\lam_\gam-\del)K_n}{2}.
\end{equation}
Then, Taylor's expansion and Chebychev's inequality gives
\begin{align}
    \nn&E\left[\pb\left(\left|\frac{1}{\sqrt{T_nb_n}}\sum_{i=1}^{k_n} w_i\sum_{j\in B_{k_i}}\zeta_j\left(\ges\right)\right|\geq M_n\right)\right]\\
    \nn&\leq\frac{4}{M^2_nT_nb_n} \sum_{i=1}^{k_n} \left|\sum_{j\in B_{k_i}}\zeta_j\left(\gam^\star\right)\right|^2+\frac{4|\ges-\gam^\star|^2}{M^2_nT_nb_n}\sum_{i=1}^{k_n} \left|\sum_{j\in B_{k_i}}\p_\gam\zeta_j\left(\gam^\star\right)\right|^2\\
    \nn&+\frac{4|\ges-\gam^\star|^4}{M^2_nT_nb_n} \sum_{i=1}^{k_n} \left|\sum_{j\in B_{k_i}}\int_0^1\p_\gam^{\otimes 2} \zeta_j(\ges+s(\gam^\star-\ges)) ds\right|^2.
\end{align}
$\zeta_j(\gam^\star)$ can be decomposed as:
\begin{equation}
    \nn \zeta_j(\gam^\star)=h_n\frac{\p_\gam c_{t_{j-1}}}{c^3_{t_{j-1}}}(c_{t_{j-1}}^2-C_{t_{j-1}}^2)+\frac{\p_\gam c_{t_{j-1}}}{c^3_{t_{j-1}}}\left[h_nC^2_{t_{j-1}}-(\D_j X)^2\right].
\end{equation}
Notice that from Proposition \ref{balance}, $T_nb_n=O_p(nh_n^2)$ in the correctly specified diffusion case, and $T_nb_n=O_p(T_n)$ in the other cases, and that the function
\begin{equation}
    \nn \frac{\p_\gam c(x,\gam^\star)}{c^3(x,\gam^\star)}(c^2(x,\gam^\star)-C^2(x))
\end{equation}
is centered.
Hence 
\begin{equation}\label{Op1}
     \frac{1}{T_nb_n} \sum_{i=1}^{k_n} \left|\sum_{j\in B_{k_i}}\zeta_j\left(\gam^\star\right)\right|^2=O_p(1),
\end{equation}
is straightforward from Lemma \ref{bsum}, Lemma \ref{wlbsum}, and Lemma \ref{misse}.
From Proposition \ref{balance}, we have
\begin{equation}
    \nn \sqrt{\frac{T_n}{b_n}}(\ges-\gam^\star)=O_p(1),
\end{equation} 
and the estimates of $\mbbh^{B}_{1,n}$ and $\mbbh^{B}_{2,n}$ imply that 
\begin{align*}
    \nn &\frac{4|\ges-\gam^\star|^2}{M^2_nT_nb_n}\sum_{i=1}^{k_n} \left|\sum_{j\in B_{k_i}}\p_\gam\zeta_j\left(\gam^\star\right)\right|^2+\frac{4|\ges-\gam^\star|^4}{M^2_nT_nb_n} \sum_{i=1}^{k_n} \left|\sum_{j\in B_{k_i}}\int_0^1\p_\gam^{\otimes 2} \zeta_j(\ges+s(\gam^\star-\ges)) ds\right|^2\\
    \nn &=o_p(1).
\end{align*}
Hence we obtain \eqref{op1}.

\medskip
Let $t=t_n$ be a root of the equation
\begin{equation}\label{be}
\mbbg_{1,n}^{\textbf{B}}\left(\ges+\sqrt{\frac{b_n}{T_n}}t\right)=0,
\end{equation}
if it exists, and otherwise, $t$ be an arbitrary element of $\mbbr^{p_\gam}$. 
For such $t$, we define
\begin{equation}
    \nn \bges=\ges+\sqrt{\frac{b_n}{T_n}}t.
\end{equation}
On the set
\begin{equation}
       \nn \inf_{|t|=K_n}\left[-\frac{1}{\sqrt{T_nb_n}}\sum_{i=1}^{k_n} w_i\sum_{j\in B_{k_i}} t^\top\zeta_j\left(\ges+\sqrt{\frac{b_n}{T_n}}t\right)\right]>0,
\end{equation}
the continuity of $\zeta(\gam)$ and \cite[Theorem 6.3.4]{OrtRhe70} ensure that a root $t$ of \eqref{be} does exist within $|t|\leq K_n$.
Since $\sqrt{b_n/T_n}K_n=o_p(1)$, \eqref{op1} and the consistency of $\ges$ lead to 
\begin{equation}\label{exgam}
    \pb(\ges^{\textbf{B}}\in\Theta_\gam)=1-o_p(1).
\end{equation}
Finally, Chebyshev's inequality, \eqref{L2ignore} and \eqref{op1} yield that for any $M>0$,
\begin{align}
    \nn &\pb\left(    \left|\sqrt{\frac{T_n}{b_n}}\frac{1}{T_n}\p_\gam\mbbg_{1,n}(\ges)(\bges-\ges)-\frac{1}{\sqrt{T_nb_n}}\sum_{i=1}^{k_n}(w_i-1) \sum_{j\in B_{k_i}} \zeta_{j}(\ges)\right|>M\right)\\
    \nn &\leq \pb\left(\inf_{|t|=K_n}\left[-\frac{1}{\sqrt{T_nb_n}}\sum_{i=1}^{k_n} w_i\sum_{j\in B_{k_i}} t^\top\zeta_j\left(\ges+\sqrt{\frac{b_n}{T_n}}t\right)\right]\leq0\right)\\
    \nn &+\pb\left(\left\{\inf_{|t|=K_n}\left[-\frac{1}{\sqrt{T_nb_n}}\sum_{i=1}^{k_n} w_i\sum_{j\in B_{k_i}} t^\top\zeta_j\left(\ges+\sqrt{\frac{b_n}{T_n}}t\right)\right]>0\right\}\cap \left\{\sup_{|t|\leq K_n}|\mbbh^B_n(t)|>M\right\}\right)\\
    \nn&\leq \pb\left(\inf_{|t|=K_n}\left[-\frac{1}{\sqrt{T_nb_n}}\sum_{i=1}^{k_n} w_i\sum_{j\in B_{k_i}} t^\top\zeta_j\left(\ges+\sqrt{\frac{b_n}{T_n}}t\right)\right]\leq0\right)+\frac{1}{M^2}\eb\left[\sup_{|t|\leq K_n}|\mbbh^B_n(t)|^2\right]\\
    \label{aequiv}& =o_p(1).
\end{align}
Combined with the estimates of $\mbbh_{1,n}^B$ and $\mbbh_{2,n}^B$, we obtain \eqref{seg1} and \eqref{seg2}.

\medskip
\subsubsection{Drift part}

Since the route is a similar to \eqref{seg1} and \eqref{seg2}, we sometimes omit the details below.
Introduce the function $\mbbu_n^B$ on $\mbbr^{p_\al}$ by
\begin{equation}
    \nn \mbbu_n^B(t)=\frac{1}{\sqrt{T_n}}\sum_{i=1}^{k_n} w_i\sum_{j\in B_{j}}\left\{\eta_j\left(\aes+\frac{t}{\sqrt{T_n}} ,\ges\right)-\eta_j\left(\aes,\ges\right)\right\}-\frac{1}{T_n}\p_\al\mbbg_{2,n}(\aes)t,
\end{equation}
and Taylor's expansion gives
\begin{align}
    \nn \mbbu_n^B(t)=\mbbu^B_{1,n}(t)+\mbbu^B_{2,n}(t)+\mbbu^B_{3,n}(t),
\end{align}
where
\begin{align*}
&\mbbu^B_{1,n}(t)=\frac{1}{T_n}\sum_{i=1}^{k_n}(w_i-1)\sum_{j\in B_{k_i}}  t^\top\p_\al \eta_j(\al^\star,\ges),\\
&\mbbu^B_{2,n}(t)=\frac{1}{T_n}\sum_{i=1}^{k_n}(w_i-1)\sum_{j\in B_{k_i}} \int_0^1\p_\al^{\otimes 2} \eta_j(\aes+s(\al^\star-\aes),\ges) ds[t,\aes-\al^\star],\\
&\mbbu^B_{3,n}(t)=\frac{1}{T_n^{3/2}}\sum_{i=1}^{k_n}w_i \sum_{j\in B_{k_i}} \int_0^1 \p_\al^{\otimes 2} \eta_j\left(\aes+\frac{1}{\sqrt{T_n}}ut,\ges\right) du[t,t].
\end{align*}
As can be seen in the proof of \eqref{seg1}, it is sufficient for \eqref{sea1} to show that 
\begin{align}
    &\label{aL2ignore} \eb\left[\sup_{|t|\leq K_n}\left|\mbbu^B_{n}(t)\right|^2\right]=o_p(1),\\
    &\label{aOp1}  \frac{1}{T_n} \sum_{i=1}^{k_n} \left|\sum_{j\in B_{k_i}}\eta_j\left(\al^\star,\ges\right)\right|^2=O_p(1),
\end{align}
where $(K_n)$ denotes the same positive sequence as the previous part.
We first show \eqref{aL2ignore}. 
Decompose $\mbbu^B_{1,n}(t)$ as
\begin{align}
    \nn \mbbu^B_{1,n}(t)&=\frac{1}{T_n}\sum_{i=1}^{k_n}(w_i-1)\sum_{j\in B_{k_i}}  t^\top\p_\al \eta_j(\al^\star,\gam^\star)\\
    \nn &+\frac{1}{T_n}\sum_{i=1}^{k_n}(w_i-1)\sum_{j\in B_{k_i}}  \int_0^1\p_\gam\p_\al \eta_j(\al^\star,\ges+u(\gam^\star-\ges))du[t,\ges-\gam^\star].
\end{align}
Since $\p_\al\eta_j(\al^\star,\gam^\star)$ can be rewritten as:
\begin{align}
    \nn \p_\al\eta_j(\al^\star,\gam^\star)&=\frac{\p^{\otimes2}_\al a_{t_{j-1}}}{c^2_{t_{j-1}}}\intj (A_s-a_{t_{j-1}})ds+\frac{\p^{\otimes2}_\al a_{t_{j-1}}}{c^2_{t_{j-1}}}\intj C_{s-}dZ_s\\
    \nn &+h_n\left[\frac{\p_\al^{\otimes2} a_{t_{j-1}}}{c^2_{t_{j-1}}}a_{t_{j-1}}+\frac{\p_\al^{\otimes2} a_{t_{j-1}}-\left(\p_\al a_{t_{j-1}}\right)^{\otimes 2}}{c^2_{t_{j-1}}}\right],
\end{align}
it follows from Jensen's inequality and Lemma \ref{bsum} that  
\begin{align}
    \nn&\eb\left[\sup_{|t|\leq K_n}\left|\frac{1}{T_n}\sum_{i=1}^{k_n}(w_i-1)\sum_{j\in B_{k_i}}  t^\top\p_\al \eta_j(\al^\star,\gam^\star)\right|^2\right]
    =O_p(K_n^2k_n^{-1})=o_p(1).
\end{align}
Again applying Jensen's inequality, we obtain
\begin{align}
    \nn & \eb\left[\sup_{|t|\leq K_n} \left|\frac{1}{T_n}\sum_{i=1}^{k_n}(w_i-1)\sum_{j\in B_{k_i}}  \int_0^1\p_\gam\p_\al \eta_j(\al^\star,\ges+u(\gam^\star-\ges))du[t,\ges-\gam^\star]\right|^2\right]\\
    \nn & \lesssim \frac{K_n^2\left|\ges-\gam^\star\right|^2}{T_n^2}\sum_{i=1}^{k_n}\left|\sum_{j\in B_{k_i}}\int_0^1\p_\gam\p_\al \eta_j(\al^\star,\ges+u(\gam^\star-\ges))du\right|^2\\
    \nn & \leq \frac{K_n^2\left|\ges-\gam^\star\right|^2}{T_n^2}\sum_{i=1}^{k_n}\sup_{\gam\in\Theta}\left|\sum_{j\in B_{k_i}}\p_\gam\p_\al \eta_j(\al^\star,\gam)\right|^2.
\end{align}
For any $q>p_\gam$, Sobolev's inequality (cf. \cite{Ada73}) gives 
\begin{align}
     \nn & E\left[\sup_{\gam\in\Theta}\left|\sum_{j\in B_{k_i}}\p_\gam\p_\al \eta_j(\al^\star,\gam)\right|^q\right]\\
     \nn &\lesssim \sup_{\gam\in\Theta}E\left[\left|\sum_{j\in B_{k_i}}\p_\gam\p_\al \eta_j(\al^\star,\gam)\right|^q\right]+\sup_{\gam\in\Theta}E\left[\left|\sum_{j\in B_{k_i}}\p_\gam^{\otimes2}\p_\al \eta_j(\al^\star,\gam)\right|^q\right].
\end{align}
Now we focus on the first term of the right-hand-side.
$\p_\gam\p_\al \eta_j(\al^\star,\gam)$ can be decomposed as:
\begin{equation}
    \nn \p_\gam\p_\al \eta_j(\al^\star,\gam)=\p_\gam c^{-2}_{t_{j-1}}(\gam)\p_\al^{\otimes 2} a_{t_{j-1}}\left(\intj A_sds+\intj C_{s-}dZ_s\right)-h_n\p_\gam c^{-2}_{t_{j-1}}(\gam)\left(\p_\al a_{t_{j-1}}\right)^{\otimes 2}.
\end{equation}
Jensen's inequality gives
\begin{align}
    \nn & \sup_{\gam\in\Theta_\gam}E\left[\left|\sum_{j\in B_{k_i}}\left[\p_\gam c^{-2}_{t_{j-1}}(\gam)\p_\al^{\otimes 2} a_{t_{j-1}}\intj A_sds-h_n\p_\gam c^{-2}_{t_{j-1}}(\gam)\left(\p_\al a_{t_{j-1}}\right)^{\otimes 2}\right]\right|^q\right]\lesssim \frac{T_n^q}{k_n^q}.
\end{align}
Burkholder's inequality leads to 
\begin{align}
    \nn & \sup_{\gam\in\Theta_\gam}E\left[\left|\sum_{j\in B_{k_i}}\p_\gam c^{-2}_{t_{j-1}}(\gam)\p_\al^{\otimes 2} a_{t_{j-1}}\intj C_{s-}dZ_s\right|^q\right]\\
    \nn &= \sup_{\gam\in\Theta_\gam}E\left[\left|\int_{(i-1)c_nh_n}^{ic_nh_n} \sum_{j\in B_{k_i}}\p_\gam c^{-2}_{t_{j-1}}(\gam)\p_\al^{\otimes 2} a_{t_{j-1}}C_{s-}\chi_j(s)dZ_s\right|^q\right]\\
    \nn & \lesssim \frac{T_n^{q/2-1}}{k_n^{q/2-1}} \sup_{\gam\in\Theta} \sum_{j\in B_{k_i}}\intj E\left[\left|\p_\gam c^{-2}_{t_{j-1}}(\gam)\p_\al^{\otimes 2} a_{t_{j-1}}C_{s}\right|^q\right]ds\\
    \nn & \lesssim \frac{T_n^{q/2}}{k_n^{q/2}},
\end{align}
so that for any $q>p_\gam$,
\begin{align}
    \nn  E\left[\sup_{\gam\in\Theta}\left|\sum_{j\in B_{k_i}}\p_\gam\p_\al \eta_j(\al^\star,\gam)\right|^2\right]\leq E\left[\sup_{\gam\in\Theta}\left|\sum_{j\in B_{k_i}}\p_\gam\p_\al \eta_j(\al^\star,\gam)\right|^q\right]^{2/q}\lesssim \frac{T_n^2}{k_n^2}.
\end{align}
Analogously, we can evaluate $\sup_{\gam\in\Theta}E\left[\left|\sum_{j\in B_{k_i}}\p_\gam^{\otimes2}\p_\al \eta_j(\al^\star,\gam)\right|^q\right]$, and the tightness of $\sqrt{T_n}(\ges-\gam^\star)$ leads to
\begin{align}
    \nn &\eb\left[\sup_{|t|\leq K_n} \left|\frac{1}{T_n}\sum_{i=1}^{k_n}(w_i-1)\sum_{j\in B_{k_i}}  \int_0^1\p_\gam\p_\al \eta_j(\al^\star,\ges+u(\gam^\star-\ges))du[t,\ges-\gam^\star]\right|^2\right]\\
    \nn &=O_p\left(K_n^2k_n^{-1}T_n^{-1}\right).
\end{align}
Hence we obtain
\begin{equation}
    \nn\eb\left[\sup_{|t|\leq K_n}\left|\mbbu^B_{1,n}(t)\right|^2\right]=o_p(1).
\end{equation}
By taking a similar route, it is easy to see
\begin{equation}
    \nn \eb\left[\sup_{|t|\leq K_n}\left|\mbbu^B_{2,n}(t)\right|^2\right]+\eb\left[\sup_{|t|\leq K_n}\left|\mbbu^B_{3,n}(t)\right|^2\right]=o_p(1),
\end{equation}
and in turn we get \eqref{aL2ignore}.
We next show \eqref{aOp1}.
Taylor's formula leads to
\begin{align}
    \nn & \frac{1}{T_n} \sum_{i=1}^{k_n} \left|\sum_{j\in B_{k_i}}\eta_j\left(\al^\star,\ges\right)\right|^2\\
    \label{Taylor} & \lesssim \frac{1}{T_n} \sum_{i=1}^{k_n} \left|\sum_{j\in B_{k_i}}\eta_j\left(\al^\star,\gam^\star\right)\right|^2+ \frac{|\ges-\gam^\star|^2}{T_n} \sum_{i=1}^{k_n} \left|\sum_{j\in B_{k_i}} \int_0^1 \p_\gam \eta_j\left(\al^\star,\ges+u(\gam^\star-\ges)\right)du\right|^2,
\end{align}
and the first term of the right-hand-side is $O_p(1)$ from an easy application of Lemma \ref{bsum}.
As for the second term of the right-hand-side, it is $O_p(k_n^{-1})$ by making use of the same argument based on Sobolev's inequality presented above.
Hence \eqref{aOp1} follows, and by mimicking the proof of \eqref{op1}, we get
\begin{align*}
    &\pb\left(\inf_{|t|=K_n}\left[-\frac{1}{\sqrt{T_n}}\sum_{i=1}^{k_n} w_i\sum_{j\in B_{k_i}} t^\top\eta_j\left(\aes+\frac{t}{\sqrt{T_n}},\ges\right)\right]>0\right)=1-o_p(1).
\end{align*}
Again by using \cite[Theorem 6.3.4]{OrtRhe70}, it turns out that the equation
\begin{equation}
    \mbbg_{2,n}^\textbf{B}\left(\aes+\frac{t}{\sqrt{T_n}}\right)=0,
\end{equation}
has a root $t$ within $|t|\leq K_n$ on the set 
\begin{equation*}
    \inf_{|t|=K_n}\left[-\frac{1}{\sqrt{T_n}}\sum_{i=1}^{k_n} w_i\sum_{j\in B_{k_i}} t^\top\eta_j\left(\aes+\frac{t}{\sqrt{T_n}},\ges\right)\right]>0.
\end{equation*}
Hence, for $\baes:=\ges+t/\sqrt{T_n}$, we get
\begin{equation*}
    \pb(\baes\in\Theta_\al)=1-o_p(1).
\end{equation*} 
Mimicking \eqref{aequiv}, we get \eqref{sea1}.
It remains to prove \eqref{sea2}, but it automatically follows from the estimates of \eqref{Taylor}.
Combined with \eqref{seg1} and \eqref{seg2}, we obtain \eqref{se2}.
Moreover, Taylor's expansion and the calculation up to here lead to
\begin{equation}\label{se1}
     \hat{A}_n\bar{\Gam}_n^{1/2}(\hat{\theta}_{n}^{\textbf{B}}-\tes)=\hat{B}_n\sum_{i=1}^{k_n} (w_i-1)\sum_{j\in B_{k_i}}\begin{pmatrix}\displaystyle\frac{\p_\gam c_{t_{j-1}}}{c^3_{t_{j-1}}}\left[h_nc^2_{t_{j-1}}-(\D_j X)^2\right] \\ \displaystyle\frac{\p_\al a_{t_{j-1}}}{c^2_{t_{j-1}}}\left(\D_j X-h_na_{t_{j-1}}\right)
\end{pmatrix}+r_{nB}.
\end{equation}

\medskip
Now we move to the proof of \eqref{secsd}, \eqref{semsd}, \eqref{sepjl}, and \eqref{sempjl}.
In the correctly specified case, by taking Lemma \ref{bsum} and Lemma \ref{wlbsum} into consideration, \eqref{secsd} and \eqref{sepjl} are trivial from \eqref{se1}.
Concerning the misspecified case, it is enough for \eqref{semsd} and \eqref{sempjl} to show that for each $l\in\{1,2\}$,
\begin{align}
&\label{pebi}\frac{1}{\sqrt{T_n}}\sum_{i=1}^{k_n}(w_i-1)\left(f_{l,{ic_nh_n}}-f_{l,{[(i-1)c_n+1]h_n}}\right)=r_{nB},\\
&\label{epebi}\frac{1}{\sqrt{T_n}}\sum_{i=1}^{k_n}(w_i-1)\left(g_{l,{ic_nh_n}}-g_{l,{[(i-1)c_n+1]h_n}}\right)=r_{nB}.
\end{align}
Since $f_1$ and $f_2$ is at most polynomial growth, it follows that for each $l\in\{1,2\}$, there exist positive constants $C_1$ and $C_2$ satisfying
\begin{equation*}
\max_{i\in\{1,\dots,k_n\}}E\left[\left|f_{l,ic_nh_n}-f_{l,[(i-1)c_n+1]h_n}\right|^2\right]\leq C_1+\sup_{t\geq 0}E[|X_t|^{C_2}]<\infty.
\end{equation*}
Hence we have
\begin{align*}
&\eb\left[\frac{1}{T_n}\sum_{i=1}^{k_n}(w_i-1)^2\left|f_{1,ic_nh_n}-f_{1,[(i-1)c_n+1]h_n}\right|^2\right]\\
&=\frac{1}{T_n}\sum_{i=1}^{k_n}\left|f_{1,ic_nh_n}-f_{1,[(i-1)c_n+1]h_n}\right|^2\\
&=O_p\left(\frac{k_n}{T_n}\right),
\end{align*}
and \eqref{pebi}.
Since $g_1$ and $g_2$ are also polynomial growth from their weighted H\"{o}lder continuity, \eqref{epebi} can be deduced in the same way. 

It remains to prove \eqref{cip}.
However, \eqref{cip} follows from $\mcf$-conditional Lindeberg-Feller central limit theorem by making use Lemma \ref{Wl2c}, and Lemma \ref{pLl2c}.
Hence the proof is complete.

\subsection*{Acknowledgement}
This work was supported by JST CREST Grant Number JPMJCR14D7, Japan, and  JSPS KAKENHI Grant Number JP19K20230.

\bibliographystyle{abbrv}

\end{document}